\documentclass{article}
\usepackage[utf8]{inputenc} 
\usepackage[T1]{fontenc}
\usepackage{textcomp} 
\usepackage[english]{babel}
\usepackage[autolanguage]{numprint} 
\usepackage{hyperref} 
\usepackage{graphicx} 
\usepackage{verbatim} 
\usepackage[all]{xy}
\usepackage{dsfont}
\usepackage{url}
\usepackage{graphicx}
\usepackage{amsmath,amssymb,amsthm} 
\usepackage{lmodern}
\usepackage{ tipa }
\usepackage{latexsym}
\usepackage{dsfont}
\usepackage{mathtools}
\usepackage{relsize}
\usepackage{exscale}
\usepackage[nottoc, notlof, notlot]{tocbibind}
\usepackage{shadowtext}
\title{Dynamical Prisoner's Dilemma : A probabilistic framework}
\author{Sylvain Gibaud}

\newcommand{\dro}{\partial_t}  

\newcommand{\R}{\mathbb{R}} 
\newcommand{\E}{\mathbb{E}} 
\newcommand{\N}{\mathbb{N}}

\newcommand{\G}{\mathbb{G}}

\newcommand{\V}{\mathbb{V}} 
\newcommand{\PP}{\mathbb{P}}

\newcommand{\Z}{\mathbb{Z}}

\theoremstyle{plain}
\newtheorem{thm}{Theorem}[section]
\newtheorem*{thm*}{Theorem}
\newtheorem{cort}[thm]{Corollary}
\newtheorem{prop}[thm]{Proposition}
\newtheorem*{prop*}{Proposition}
\newtheorem*{prop4C*}{Proposition 4C}

\newtheorem{lemme}[thm]{Lemma}

\newtheorem*{thmFR2A*}{Théorème 2A}
\newtheorem*{thmFR2B*}{Théorème 2B}
\newtheorem*{thmFR2C*}{Théorème 2C}
\newtheorem*{thmFR3A*}{Théorème 3A}
\newtheorem*{thmFR3B*}{Théorème 3B}
\newtheorem*{thmFR3C*}{Théorème 3C}
\newtheorem*{thmFR3D*}{Théorème 3D}
\newtheorem*{thmFR4A*}{Théorème 4A}
\newtheorem*{thmFR4B*}{Théorème 4B}
\newtheorem*{thmFR4C*}{Théorème 4C}

\newtheorem{thmFR*}{Théorème}

\theoremstyle{remark}

\newtheorem*{remarque}{Remark}

\theoremstyle{definition}
\newtheorem*{idée}{Sketch of the proof}

\newtheorem{definition}[thm]{Definition}

\newtheorem*{Not}{Notations}

\begin{document}

\maketitle

\begin{abstract}
We put a probabilistic framework on the Demographic Prisoner's Dilemma. In this model, cooperating and defecting individuals are placed on a torus to move and play prisoner's dilemma game, if they are on the same site. Each individual accumulates its payoff into a quantity called wealth. If an individual becomes wealthy enough, it can have an offspring. If its wealth becomes negative, it disappears.

In this framework we prove that if if the Sucker payoff is far greater than the Reward then for all initial state almost surely all cooperators will die. Moreover if the Temptation payoff (resp. Reward) are far greater than the Punition (resp. Sucker payoff) then for all initial state with positive probability cooperators and defectors live \emph{ad vitam eternam}. We also set a Mean Field model on the demographic prisoner's dilemma and prove on a linearized version of the Mean Field model that with weaker assumptions with positive probability Cooperators live \emph{ad vitam eternam}.
\end{abstract}

\tableofcontents
\section{Introduction}
Playing particle systems as a modeling tool in the evolution literature have been introduced by \cite{smith1973lhe}. Since, many ecological models have used playing particle system's approach (see a list of examples in the book of \cite{hartl1997principles}). Yet a lot of the literature on the evolutionary games have been focused on random matching over finites games models \cite{boylan1992LawLargeNumberRandom,ellison1994cooperation,weibull1997evolutionary}. Then motivated by biological applications, spatiality have been introduced in models \cite{nowak1993spatial,szabo1998evolutionary,hauert2002effects}. In those previous models, individuals cannot move and evolution is managed by a Wright-Fisher generation system (\emph{i.e.} at each step of the evolution one whole generation gives birth to another generation and dies).\\
\cite{epstein1998zones} introduced a model where individuals can move and are not synchronized (they die at different times and gives birth at different times). This model is called Demographic Prisoner's dilemma. Before explaining the model and a bit of its probabilistic framework, let us give an intuition inspired by the article of \cite{turner1999prisoner}. \\
You put viruses on a torus. Each virus has embedded in his RNA one of the two following behaviors: either she manufactures diffusible (shared) intracellular product, either she sequesters it and take advantage of the virus producing the product. Each virus can move on a fixed torus. They also have energy such that when they have intracellular product they get energy and when they manufacture it or sequester nothing they lose energy. A virus without energy dies. A virus with a lot of energy can split and bring a child (with its RNA) in the torus. \\
\begin{Not}\textcolor{white}{text}

We denote by $\N$ the set of non negative integers and $\N^*$ the set of positive integers.

\end{Not}\textcolor{white}{text}\\
The model is:
\begin{enumerate}
		\setlength\itemsep{0.2em}
		\item Let $(\Z / m \Z)^2$ be a fixed torus with $m \in \N^*$. It is the space where individuals move.
		\item We make the assumption that the torus cannot bear an infinite number of individuals. Let us call $K \in \N^*$ the maximum number of individuals in the torus. 
		\item The individuals move according to continuous time independent symmetric simple random walks. At random times (of rate $d\in \R_+^*$) an individual moves from his position to one of its nearest neighbors with equal probability, as shown in the following example of evolution.
			\begin{figure}[h!]
			\caption{Example of moving of particles}\label{Figure example moving}
			\begin{center}
					\includegraphics[scale=0.25]{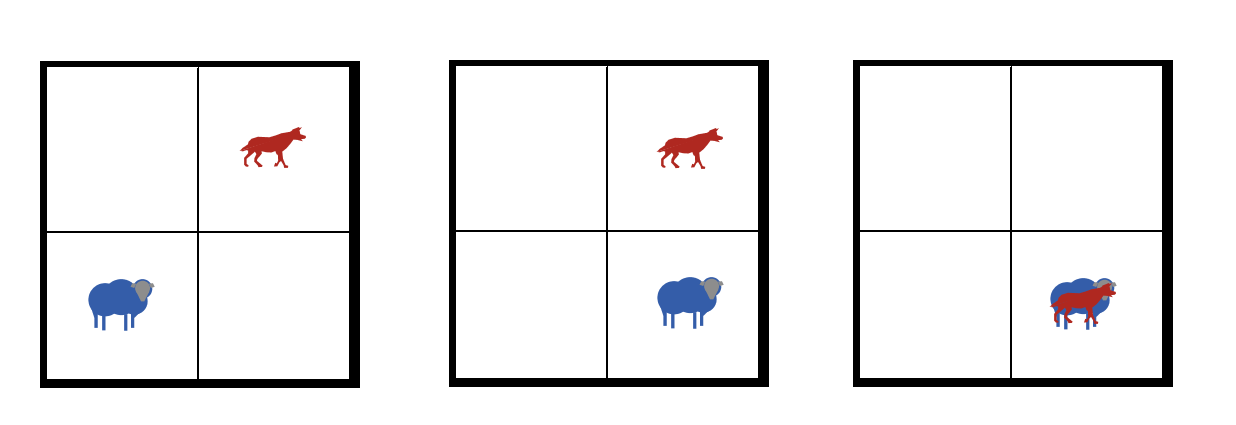}
				\end{center}
				\end{figure}
		\item Each individual has a wealth. If the wealth becomes non positive, the individual dies (and then stop playing with the other individuals).
		\item   The wealth changes through games. The game is the prisoner's dilemma:
		\begin{itemize}
			\setlength\itemsep{0.2em}
			\item  The players have two actions to \textit{Cooperate} or to \textit{Defect}.
			\item If both \textit{Cooperate} they get a Reward $\mathbf{R}$. But if one of the two {defects}, the \textit{Defector} gets a Temptation payoff $\mathbf{T}$, and this payoff is bigger than the Reward. If a \textit{Cooperator} is being \textit{Defected} instead of getting a reward he gets a Sucker payoff $\mathbf{-S}$. If both players \textit{Defect} (let us call them Player 1 and Player 2)  the nature flips a coin, if it is head Player 1 gets a Punishment payoff of $-2 \mathbf{P}$ and Player 2 gets a payoff of 0, if it is tail Player 2 gets a Punishment payoff of $-2 \mathbf{P}$ and Player 1 gets a payoff of 0.
			\item Then the payoff satisfies $\mathbf{T}>\mathbf{R}>0>$ and $\mathbf{S}>\mathbf{P}>0$. This is summarized in the following payoff matrix (action Top and action Left are \textit{Cooperate}, action Bottom and action Right are \textit{Defect}).\\
				
				\[\left(\begin{array}{cc}
				(\mathbf{R},\mathbf{R}) & (-\mathbf{S},\mathbf{T}) \\ (\mathbf{T},-\mathbf{S}) & {(P_1,P_2) }  
				\end{array}\right),
				\]
				where $(P_1,P_2)$ is a random variable with distribution $\frac{1}{2}\delta_{(-2\mathbf{P},0)} + \frac{1}{2}  \delta_{(0,-2\mathbf{P})}$
				
			\item Playing \textit{Cooperate} is strictly dominated by playing \textit{Defect}.
			\item To determine the action played by the individuals we place ourself in the framework where: "Each individual plays only one action, either he will \textit{Defect} every time either he will \textit{Cooperate} every time". Then a particle has a fixed action (\textit{Cooperate} or \textit{Defect}). There are two kinds of particles: the ones who always \textit{Cooperate} and the ones who always \textit{Defect}.
			\item At the end of each game the wealth are updated, adding the respective payoff of the game played.
			\item To make the games happen, each couple of individuals is given a Poisson process independent of everything of parameter $v$. When this Poisson process realizes, if the individuals are on the same site and if their wealths are positive the individuals play together. Otherwise nothing happens.
		\end{itemize} 		
		\item Each individual is given a Poisson process of parameter $b>0$. When this birth Poisson process realizes if the individual's wealth is more than a given threshold $w_c >0$ and if there is less than $K$ individuals on the torus then the individual gives birth to an offspring. The offspring has the same strategy as its unique parent (it \textit{cooperates} if its parent \textit{cooperates}, \textit{defects} if its parent \textit{defects}). Moreover the fixed birth wealth $0< w_0 < w_c $ of the offspring is given by its parent. That is after the birth (then the parent lose wealth, and the child begins its life with a wealth equal to $w_0$).
		
	\end{enumerate}
	
	 \cite{epstein1998zones} introduced the first Demographic Prisoner's dilemma model, and explore it doing simulations. The demographic prisoner's dilemma has been a popular model \cite{axelrod2000six,ifti2004effects,ohtsuki2006replicator} in theoretical biology. Also since he used multi-agent as a way of modeling evolutionary game theory, Epstein made an useful link between computer sciences and biology as presented by \cite{tumer2004survey}.\\

	  Yet in conclusion of his article Epstein said: "The only claims that can be advanced definitely are that this specific complex of assumptions is \textit{sufficient to generate cooperative persistence} on the timescale explored in the research. [...] Obviously it would be worthwhile to [...] if possible, assess their generality mathematically." \\
	  After this article, \cite{dorofeenko2002dynamical}, assuming statistical independence, proved the convergence of the Epstein model (when the payoff and the step of the grid tend to 0) to a reaction-diffusion process. They also studied via numerical results this reaction diffusion process. \cite{namekata2011effect} add reluctant players (Tit for Tat players where the first action is \textit{Defect}) in the model. They studied the extended model using simulations. \\
	
	The main contributions of this article is generalizing Epstein model mathematically that is: \begin{itemize}
	\item Putting a probabilistic framework on the Epstein model without doing too strong assumptions (slightly modifying the payoff matrix, making individuals moving following independent simple random walks instead of exclusion process).
	\item Proving asymptotic (on time) results on the survival of the cooperators.
	\item Introducing a Mean Field model (or well mixed)
	\item Proving that we can insure a certain level of cooperation without too strong assumption on the payoff matrix.
	\end{itemize}
	
	Let us introducing the first results.\\
	
	\begin{Not}
	\begin{itemize}
	\item We call configuration the data of the positions, the behaviors, and the wealth of every particles. Then a configuration is an element of $E := ((\Z / m \Z)^2 \times \R_+ \times \lbrace -1,0,1 \rbrace)^K$ where $\lbrace -1,0,1 \rbrace$ indicates the behavior of a particle.
	\item Since Markov processes are defined by a collection of probability distribution on the space of right continuous, left limited functions from $\R_+$ to $E$ (this space is denoted $D(\R_+,E)$) indexed by $E$. We denote the distributions of the previous model $(\PP_\sigma)_{\sigma \in E}$ ($\sigma$ represents the initial state). 
	\end{itemize}
	\end{Not}
	
	\begin{thm*}\label{Intro thm Extinction presque sure}
	
	There exists a constant $\mu >0$ depending only on $v,d,K,b,w_0,w_c$ and $m$ such that if: 
	\[
	\mu \mathbf{R} < \mathbf{S}
	\]
	then for each initial configuration $\sigma$ with at least one defector with positive wealth: \\
	\[\PP_\sigma(\lbrace\text{Eventually all the cooperators will be dead}\rbrace) = 1 .\]
	\end{thm*}
	\textcolor{white}{text}\\
	\begin{thm*}\label{Intro thm Coexistence ad vitam eternam Spatial}
	There exists a constant $\nu,\nu' > 0$ depending only on $v,d,K,b,w_0,w_c$ and $m$ such that if: 
	
\begin{align*}
	\nu (\mathbf{S}+w_0) < \mathbf{R}
	\\
	\nu' (2\mathbf{P} + w_0) < \mathbf{T}
\end{align*}
	
	then for each initial configuration $\sigma$ with at least two cooperators and one defector with positive wealth: \\
	
	\[
	\PP_\sigma(\lbrace \text{The cooperators and defectors present in the beginning never die} \rbrace) > 0.
	\]
	\end{thm*}
	
	Those are qualitative results. The first theorem says that when the Sucker payoff is far bigger than the Reward payoff, then with any initial configuration (except the trivial ones) the cooperators die almost surely. The second theorem says that if the Reward is far bigger than the sum of the Sucker payoff and the birth wealth then from every non trivial configuration with positive probability the cooperators will never die. \\
	The two theorems talk of the red and blue area (or the areas pointed by the arrows) on the following graph. 
	In order to understand better the behavior of this system, we draw some simulations with the following data:
	\begin{itemize}
		\item The size of the torus is $m = 7$.
		\item There are initially 10 cooperators and 10 defectors.
		\item Initially all individuals have a wealth of $10$.
		\item The rate of the game Poisson processes is the same than the rate of the moving Poisson processes $d=v=5$.
		\item The Temptation payoff is $\mathbf{T}=\mathbf{R}+1$ and the Punishment payoff is $\mathbf{P}=\mathbf{S} - 1$.
		\item The stopping time of the simulation is 10,000 moves, games or birth Poisson process realizations. (Usually in the previous simulations this density becomes nearly constant after 3,000 realizations of Poisson processes).
		\item The critical wealth (wealth that allows to give birth) is: $w_c =10$.
		\item The birth wealth for a child is: $w_0 = 3$.
		\item The maximum number of particles $K = 10^m=10^7$.
	\end{itemize}
	
	\textit{Method}\\
	We make a batch of 100 simulations of the system for each couple of payoff $(\mathbf{R},\mathbf{S})$ ($\mathbf{R},\mathbf{S}\in \lbrace0,\dots,100\rbrace$). At the end of each simulation we measure the number of cooperators with a positive wealth. After drawing these simulations we take the average of this measure over all these 100 simulations. Finally we make the figure saying that the color red corresponds to a small number of cooperators with a positive wealth at the end of the simulation, and the color blue corresponds to a high number of cooperators with a positive wealth at the end of the simulation (dark red correspond to 0 cooperator with a positive wealth at the end of the simulation).\\

	\begin{figure}[h]
	\caption{Survival of cooperators in function of the Reward and the Sucker payoffs}\label{Figure Diagramme de phase 3D avec Naissance}
	\begin{center}
	\includegraphics[scale=0.3123]{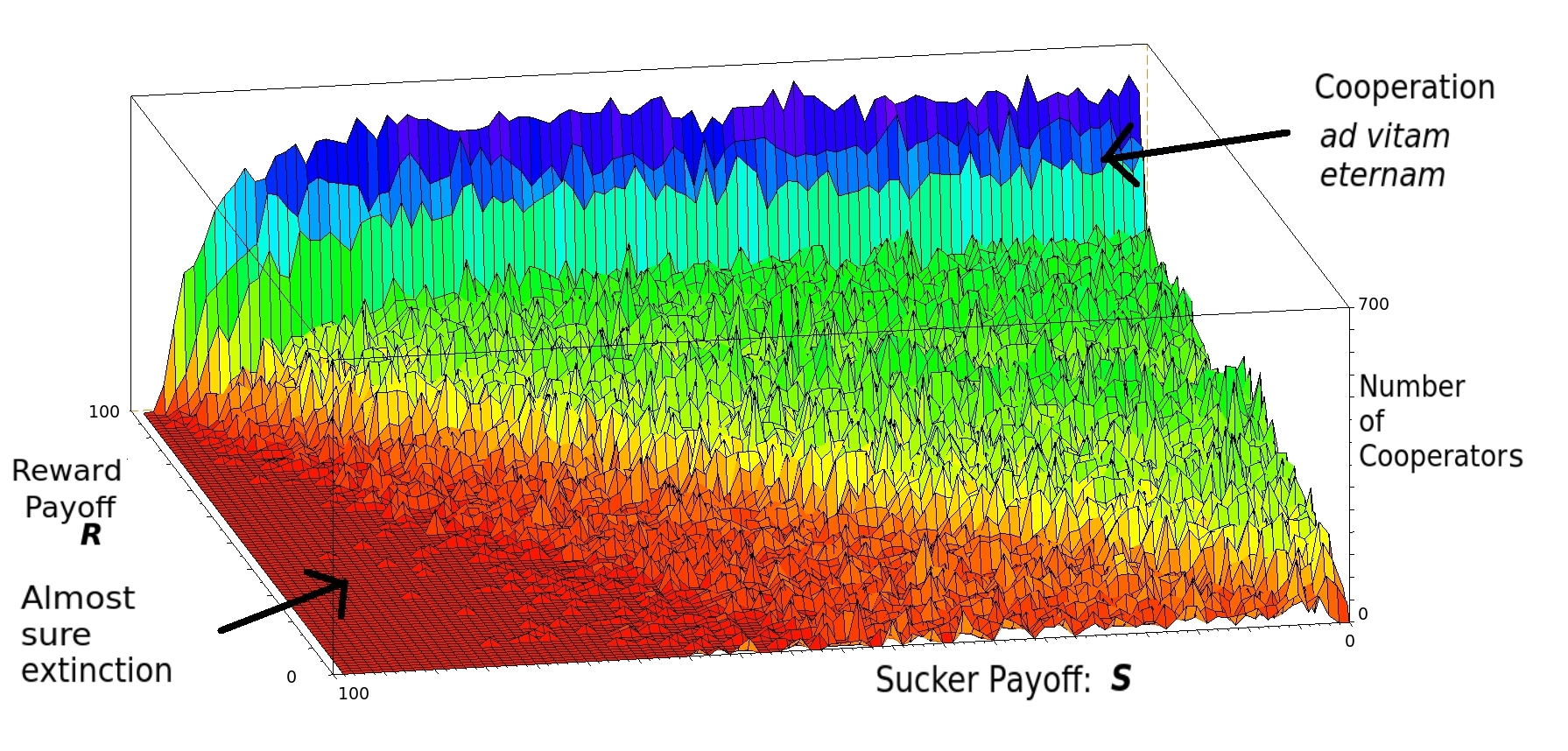}
	\end{center}
	\end{figure}
In order to understand better the area in green and orange (intermediate area between the two pointed areas) we will consider the following Mean Field model:
\begin{itemize}
\item particles stop giving birth.
\item there is no spatial condition.
\item instead of considering a linear interaction with many particles, we consider one defector and one cooperator with non linear interaction (the evolution depends also on the distribution of the process).
\item all particles begin with the same distribution of wealth.
\end{itemize}
The Mean Field system is the Markov process $(\mathfrak{C}(t),\mathfrak{D}(t))_t$. We will look at the distribution of $(\mathfrak{C}(t))_t$. The intuition of $\PP(\mathfrak{C}(t) = w) = p$ is: at time $t$, $p\%$ of the population of cooperators has a wealth $w$. We define $(\mathfrak{C}(t),\mathfrak{D}(t))_t$ by: (with $\beta_0+\rho_0 = 1$ fixed) and an initial measure $m_0$
\begin{align*}
\mathfrak{D}(t_{n+1}^\mathfrak{D}) - \mathfrak{D}(t_{n}^\mathfrak{D}) = \mathds{1}_{\mathfrak{D(t_n^\mathfrak{D})>0}} \,U_\mathfrak{D}(t^\mathfrak{D}_{n+1})\\
\mathfrak{C}(t_{n+1}^\mathfrak{C}) - \mathfrak{C}(t_{n}^\mathfrak{C}) = \mathds{1}_{\mathfrak{C(t_n^\mathfrak{C})>0}} \,U_\mathfrak{C}(t^\mathfrak{C}_{n+1})
\end{align*}
where:
\begin{itemize}
	\item[-] $(t_n^\mathfrak{D})_n$ and $(t_n^\mathfrak{C})_n$ are sequences of Poisson times of intensity $\frac{v}{2}$ (\emph{i.e. } for example with $\mathcal{N}^\mathfrak{D}$ a Poisson process of parameter $\frac{v}{2}$ for all $n\in \N^*$ $t_n^\mathfrak{D}$ is defined by $t_n^\mathfrak{D} = \inf (t>0 / \mathcal{N}^\mathfrak{D}_t \geq n)$).
	\item[-] for $I = \lbrace t_n^\mathfrak{C} ,n \in \N \rbrace \cup \lbrace t_n^\mathfrak{D},n\in \N \rbrace$, all $t \in I$ and $U_\mathfrak{D}(t)$ and $U_\mathfrak{C}(t)$ are random variables independent from everything such that: 
	\[
	U_\mathfrak{D}(t) = \left\{\begin{array}{rl}
	 - 2 \mathbf{P} & \text{with probability } \frac{1}{2}\rho^0\PP(\mathfrak{D}(t)>0)\smallskip\\
	 \mathbf{T} & \text{with probability } \beta^0\PP(\mathfrak{C}(t)>0)\smallskip\\
	 0 & \text{ with probability } 1 - \beta^0\PP(\mathfrak{C}(t)>0) - \frac{1}{2}\rho^0\PP(\mathfrak{D}(t)>0)
	\end{array}\right.
	\]
 and 
	\[
	U_\mathfrak{C}(t) = \left \lbrace\begin{array}{rl}
	 -  \mathbf{S} & \text{with probability } \rho^0\PP(\mathfrak{D}(t)>0)\smallskip\\
	 \mathbf{R} & \text{with probability } \beta^0\PP(\mathfrak{C}(t)>0)\smallskip\\
	 0 & \text{ with probability } 1 - \beta^0\PP(\mathfrak{C}(t)>0) - \rho^0\PP(\mathfrak{D}(t)>0)
	\end{array}\right.
	\]
\end{itemize}
The intuition behind these formulas is that we update (for example $\mathfrak{C}(t)$) following a Poisson process of intensity $\frac{v}{2}$. If $\mathfrak{C}>0$ we update it doing: 
\[
\mathfrak{C}(t) \leftarrow \left \lbrace\begin{array}{rl}
	\mathfrak{C}(t) -  \mathbf{S} & \text{with probability } \rho^0\PP(\mathfrak{D}(t)>0)\smallskip\\
	\mathfrak{C}(t) + \mathbf{R} & \text{with probability } \beta^0\PP(\mathfrak{C}(t)>0)\smallskip\\
	\mathfrak{C}(t) & \text{ with probability } 1 - \beta^0\PP(\mathfrak{C}(t)>0) - \rho^0\PP(\mathfrak{D}(t)>0)
	\end{array}\right.
\]

An important remark is that the law of the wealth of any cooperator, resp any defector, (as a stochastic process) in the spatial model converges in finite dimensional distribution to $\mathfrak{C}$, resp $\mathfrak{D}$ when $m^2=N$ and $d\to +\infty$ then when $N\to +\infty$.\\
This distribution has three principles:
\begin{itemize}
\item A drift (positive or negative)
\item A diffusion
\item An absorbing bound on 0 causing non linearity.
\end{itemize}

\begin{figure}[h!]\label{Fig Dist MF Wealth DPD}
\caption{Distribution of the Mean Field Wealth}
\begin{center}
\includegraphics[scale=0.5]{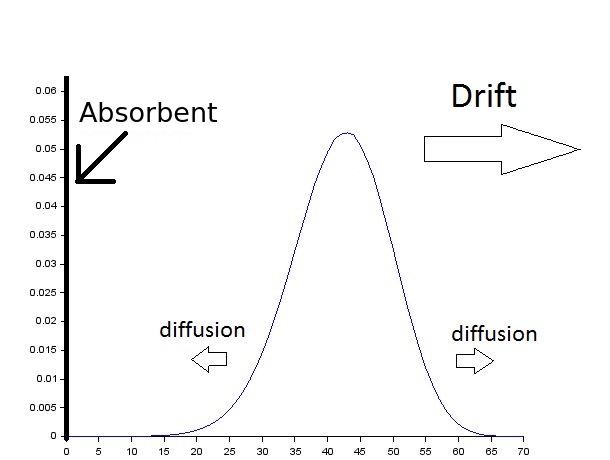}
\end{center}
\end{figure}

Hence this process is non linear. We consider a simpler process instead : a linearized process $(\overline{\mathfrak{C}}(t),\overline{\mathfrak{D}}(t))_t$.

\begin{align*}
\overline{\mathfrak{C}}(t_{n+1}^\mathfrak{C}) - \overline{\mathfrak{C}}(t_{n}^\mathfrak{C}) = \overline{U}_\mathfrak{C}(n) \\
\overline{\mathfrak{D}}(t_{n+1}^\mathfrak{D}) - \overline{\mathfrak{D}}(t_{n}^\mathfrak{D}) = \overline{U}_\mathfrak{D}(n)
\end{align*}
avec $(\overline{U}_\mathfrak{C}(n))_n$ et $(\overline{U}_\mathfrak{D}(n))_n$ two independent sequences of i.i.d. random variables such that
\[
	  	 	\overline{U}_\mathfrak{C}(n) = \left \lbrace\begin{array}{rl}
	  	 	 -  \mathbf{S} & \text{with probability } \rho^0\smallskip\\
	  	 	 \mathbf{R} & \text{with probability } \beta^0\smallskip\\
	  	 	 0 & \text{ with probability } 1 - \beta^0 - \rho^0
	  	 	\end{array}\right.
\]  et
\[
	  	 	\overline{U}_\mathfrak{D}(n) = \left \lbrace\begin{array}{rl}
	  	 	 -  2\mathbf{P} & \text{with probability } \frac{\rho^0}{2}\smallskip\\
	  	 	 \mathbf{T} & \text{with probability } \beta^0\smallskip\\
	  	 	 0 & \text{ with probability } 1 - \beta^0 - \frac{\rho^0}{2}
	  	 	\end{array}\right.
\]

On this linearized system we have the following results:

\begin{thm}\label{Intro Thm Appli MF}
Let us suppose that:
\[
\beta^0 \mathbf{R} - \rho^0 \mathbf{S} >0
\]
Then we have: for any fixed $q_0 >0$ 
\[
\overline{\mathfrak{C}}(t) \longrightarrow +\infty \qquad \PP_{q_0} a.s.
\]

\[
\PP_{q_0} (\forall t\geq 0, \overline{\mathfrak{C}}(t)>0) >0
\]
\end{thm}

We also have the following proposition about the concentration of wealth.

\begin{prop*}
For all $q_0 \in \R_+$

We denote :\[
\mathfrak{m} := \frac{v}{2}\left(\beta^0 \mathbf{R} - \rho^0\mathbf{S}\right), \qquad \sigma^2 := \frac{v}{2}\left(\beta^0 \mathbf{R}^2 + \rho^0 \mathbf{S}^2\right).
\]

Then we have : $\forall \eta>0,t\geq0$

\[\mathbb{P}_{q_0}\left( \overline{\mathfrak{C}}(t)\in \left[ q_{0}+\mathfrak{m}t-\eta \sqrt{\sigma ^{2}t},q_{0}+\mathfrak{m}t+\eta \sqrt{\sigma
^{2}t}\right] \right) \geq 1-\eta ^{-2}.\]

Moreover for all $\tau$ such that:

\[\tau <q_{0}-\frac{\eta ^{2}\sigma ^{2}}{4\mathfrak{m}},\]

then 

\[\mathbb{P}(\overline{\mathfrak{C}}(t)>\tau )\geq 1-\eta ^{-2}.\]

\end{prop*}

This article is in two parts. In the first one, we introduce the spatial model with his probabilistic framework. We also state and give the sketch of the proofs of the two qualitative theorems. In the second part, we give some reminders on Continuous Time Markov processes (in particular the theory of infinitesimal generators). We do that in order to describe and justify the Mean Field model. Then we state Mean field theorem. The proofs of the theorems are in the appendix.

\section{Spatial Model}
\begin{remarque}
	We call particle, player or individual, the agent defined in the following.
\end{remarque}
\subsection{Model}\label{Spatial Model}
Let us give a probabilistic framework for the Demographic prisoner's dilemma of \cite{epstein1998zones}. Let $(\Omega,\mathcal{T},\PP)$ be a probability space.

	\begin{enumerate}
		\setlength\itemsep{0.2em}
		\item Let  $m \in \N^*$. Let $K \in \N^*$.
		\item For a particle $i \in \lbrace 1,\dots, K \rbrace$, let us denote $X^i(t)$ its position at time $t \in \R_+$. Then $\forall i\in \lbrace 1,\dots,K \rbrace$ $X^i$ is a continuous simple symmetric random walk on $(\Z / m \Z)^2$ of rate $d\in \R_+$.
		\item We denote the wealth of particle $i \in \lbrace 1,\dots,K \rbrace$ at time $t\in \R_+$ by $Y^i(t) \in \R$.
		\item The action played by $i\in \lbrace 1,\dots,K \rbrace$ is coded in a process $Z^i$. Particle $i$ plays \textit{Cooperate} if $Z^i = 0$ and plays \textit{Defect} if $Z^i = 1$. If $Z^i = -1$ the particle does not play.
		\item   The wealth changes through games. 
		\begin{itemize}
			\setlength\itemsep{0.2em}
			\item To make the games happen, each couple of individuals $(i,j)$ is given a Poisson process independent of everything of parameter $v$. When this Poisson process realizes (for example at a time $t$), if the individuals are on the same site (\emph{i.e.} if $X^i(t)=X^j(t)$), if their wealths are positive (\emph{i.e.} if $Y^i(t)>0$ and $Y^j(t)>0$) and if $Z^i,Z^j$ are both not equal to -1 the individuals play together. Otherwise nothing happens. Then if an individual has a negative wealth, he can't play with the other individuals.
			\item A point to notice is that two players cannot lose wealth simultaneously. As a consequence, two individuals cannot kill each other.
			\item At the end of each game (between for example player $i$ and $j$) $Y^i$ and $Y^j$ are updated, adding the payoff of the game played by $i$ and $j$.
		\end{itemize} 
			\item For $i\in \lbrace 1,\dots,K \rbrace$ when $Y^i>w_c$, and when the birth Poisson process of $i$ realizes, he gives birth to an offspring. The index $j$ of the offspring is drawn uniformly randomly from $\lbrace j \in \lbrace 1,\dots,K \rbrace, Z^j = -1 \rbrace$, and then $X^j$, $Y^j$, $Z^j$ are updated doing $X^j \leftarrow X^i$, $Y^j \leftarrow w_0$ and $Z^j \leftarrow Z^i$.
	\end{enumerate}
	
	We have that $(X(t),Y(t),Z(t))_t$ is a Continuous Time Markov Chain, we denote by $(\mathcal{F}^d_t)_t$ its filtration such that $(X(t),Y(t),Z(t))_t$ is adapted to $(\mathcal{F}^d_t)_t$.\\

\begin{definition}
We call set of configuration $\left((\Z / m \Z)^2 \times \R \times \lbrace -1,0,1 \rbrace\right)^{K}:=E $. A configuration is an element usually denoted $\sigma$ of $\left((\Z / m \Z)^2 \times \R \times \lbrace -1,0,1 \rbrace\right)^{K} $.
We denote the non trivial configurations by:
\begin{itemize}
	\item $\mathcal{C}$ the set of configurations with at least 2 cooperators and 1 defector with positive wealths.
	\item $\mathcal{C}'$ the set of configurations with at least 1 cooperator and 1 defector with positive wealths.
\end{itemize}
A particle system is a Continuous Time Markov Chain with state space: the space of configuration. It is usually denoted $(X(t),Y(t),Z(t))_t:= (\sigma_t)_t$.
\end{definition}

\vspace{20pt}
Before stating the theorems, let us remind the data of the model.
 \begin{itemize}
	\item $m\in \N$ the size of the torus, state space of the position.
	\item $K \in \N^*$ is the maximum number of players on the torus.
	\item 	We fix the payoff initially $\mathbf{T}>\mathbf{R}>0>$ and $\mathbf{S}>\mathbf{P}>0$. The payoff matrix is 
				\[\left(\begin{array}{cc}
				(\mathbf{R},\mathbf{R}) & (-\mathbf{S},\mathbf{T}) \\ (\mathbf{T},-\mathbf{S}) & {(P_1,P_2) }  
				\end{array}\right),
				\]
				where $(P_1,P_2)$ is a random variable with distribution $\frac{1}{2}\delta_{(-2\mathbf{P},0)} + \frac{1}{2}  \delta_{(0,-2\mathbf{P})}$
		
	\item The position is denoted by $X$, the wealth by $Y$ and the strategy by $Z$. If $Z=0$ the particle is a cooperator, if $Z=1$ the particle is defector. If $Z = -1$ the particle cannot play (because she is not born yet).
	\item If $Y\leq 0$, the particle stop playing. If $Y>w_c>0$ the particle can give birth to a particle with initial wealth $w_0 < w_c$.
	\item $d\in \R_+$ is the rate of the Poisson process making a particle move. $v \in \R_+$ is rate of the Poisson process making a couple of particles plays together. $b \in \R_+$ is the rate of the Poisson process making the player trying to have an offspring.
\end{itemize}

\subsection{Almost sure extinction}
In this subsection we will show a first qualitative theoretic result: when the sucker payoff $\mathbf{S}$ is far greater than the reward payoff $\mathbf{R}$ then almost surely for any initial configuration (with at least one defector) cooperators will die.

\begin{thm}\label{thm Extinction presque sure}

	There exists a constant $\mu >0$ depending only on $v,d,K,b,w_0,w_c$ and $m$ such that if: 

\begin{eqnarray}
	\mu \mathbf{R} < \mathbf{S}
\end{eqnarray}
	then for each initial configuration with at least one defector $\sigma$: \\
		\[\PP_\sigma(\lbrace\text{Eventually all the cooperators will be dead}\rbrace) = 1 .\]
	
\end{thm}

\begin{idée}
The aim of the proof is showing that the total wealth of cooperators goes to 0 almost surely.
The sketch of the proof is the following: 
\begin{enumerate}
\item The only way the total cooperator wealth decreases is via games with defectors (giving birth does not change the total wealth of cooperators). Since $(\Z / m \Z)^2$ is finite, going from a configuration $\sigma_0$, the first game between a defector and a cooperator arrives at a finite random time $\tau$.  At this time the variation of total wealth is less than $ -\mathbf{S} + \mathbf{R}\mathcal{N}_{\tau}$ where $\mathcal{N}$ is the Poisson process counting the games.
\item  Then to apply a Law of Large Number argument we want a condition such that:\\
 $ \E\left(-\mathbf{S} + \mathbf{R}\mathcal{N}_{\tau}\right) < 0$. 
\item To find this condition we have to upper bound uniformly in the configurations the first time for a defector to play with a cooperator. This will be one of the main difficulty of the proof.
\begin{enumerate}
\item We remark that there exists a minimum number $\mathfrak{m}$ of realizations of Poisson processes such that for each non trivial configuration there is a Cooperator-Defector game in less than $\mathfrak{m}$ realizations.
\item Since the evolution is managed by Poisson processes, for every configuration $\sigma$ there is a Cooperator-Defector game in less than $\mathfrak{m}$ realizations of Poisson processes with probability $\varepsilon(\sigma) >0$.
\item Since $\varepsilon(\sigma)$ is the probability of a particle moving and a game happening, $\varepsilon$ does not depend on the wealth of the individuals. Since if we don't look at the wealth of the individuals there is a finite number of configurations, we take $\varepsilon$ the minimum over these configurations. This part is really detailed in the proof.
\item The probability that a Cooperator-Defector game happens after $k \mathfrak{m}$ realizations of Poisson processes is bugger than $(1 - \varepsilon)^k$.
\end{enumerate}
\item We finish the proof using Borel Cantelli's Lemma.
\end{enumerate}
\end{idée}

\subsection{Coexistence \emph{ad vitam eternam}}
In this subsection we will show a second qualitative result that is: when the reward is far greater than the sucker payoff and the initial offspring wealth then with positive probability the cooperators live until the end of times.

\begin{thm}\label{thm Coexistence ad vitam eternam Spatial}
There exists a constant $\nu,\nu' > 0$ depending only on $d, v$,$w_c$,$w_0$, $K$ and $m$ such that if: 

\begin{eqnarray}\label{Hyp thm Coop ad vitam eternam}
\nu (\mathbf{S}+w_0) < \mathbf{R} \\
\nu' (2\mathbf{P} + w_0) < \mathbf{T}
\end{eqnarray}

then for each configuration $\sigma$ with at least two cooperators and one defector: \\

\[
\PP_\sigma(\lbrace \text{the cooperators and defectors of $\sigma$ never die} \rbrace) > 0 
\]
\end{thm}
\begin{idée}
The main ideas are to consider a \textbf{Ghost system} and then to extend the result of the \textbf{Ghost system} using Burkholder-Davis-Gundy inequalities and Borel-Cantelli Lemma. We call the system introduced in the first section \textbf{True system}.
\begin{enumerate}
\item In the \textbf{Ghost system}, everybody plays even with negative wealth. Also whenever something in the system happens all individual's wealth decrease of $w_0$ (except if an individual gives birth, this particular individual does not lose additional wealth). If in the \textbf{Ghost system} we can prove that every individual always have positive wealth then this \textbf{Ghost system} is actually the \textbf{True system}.
\item Using a Burkholder-Davis-Gundy argument with Borel Cantelli's Lemma, we prove on the \textbf{Ghost system} that the minimum wealth over all individuals $\mathfrak{C}^{min}_t \to +\infty$ a.s. when $t \to +\infty$.
\item We just have to consider an event (of positive probability) such that no individual has non positive wealth until they accumulated enough wealth to make the coupling between the \textbf{Ghost system} and the \textbf{True system}.
\end{enumerate}
\end{idée}
\vspace{5pt}
Theorem \ref{thm Extinction presque sure} and Thm \ref{thm Coexistence ad vitam eternam Spatial} are qualitative results. In order to have quantitative results we consider in the following a Mean Field model. In simulation we notice that birth are really useful for maintaining the cooperation (see for example Run 3 and Run 4 of Eptein \cite{epstein1998zones}), in order to guarantee some result on the survival of cooperators in the Mean Field model we won't consider births.

\section{Mean Field model}
\subsection{Reminders on Markov process and infinitesimal generators}\label{Subsect Reminder on Markov and gene}

Let us firstly give some reminders about Markov processes.\\
One good way to study Markov processes is via their distribution. For a complete separable state space $F$, a process $X=(X(t))_{t\geq0}$ has value in the canonical space.
\begin{definition}[Canonical space]
The canonical space $D(\R_+,F)$ is the space of right continuous functions from $[0,+\infty)$ to $F$ with left limits, endowed with the Skohorod topology associated to its usual metric. With this metric, $D(\R_+,F)$ is complete and separable (see the book of Ethier Kurtz \cite{EK} for more details).\\

We denote by $(\mathcal{T}_t)_t$ the natural filtration associated to $D(\R_+,F)$.
We define the canonical process $X := (X(t))_{t\geq0}$ by:\[
\forall t\geq 0, \qquad	X(t): \omega \in D(\R_+,F) \mapsto \omega(t) \in F
	\]
\end{definition}

To define a Markov process in a discrete countable state space $F$ (which will be the case here), we need an initial measure $\nu$ on $F$ and a matrix called rate matrix $(\mathcal{A}(x,y))_{x,y\in F}$ with $\mathcal{A}(x,x) = - \sum\limits_{\substack{y\in F\\ y\neq x}} \mathcal{A}(x,y)$. For $x\neq y$ $\mathcal{A}(x,y)\geq 0$ gives the rates of transition of the future process from $x$ to $y$. With this matrix we define the generator (which is one of the main tool in the study of Markov processes) of the following Markov process such that: for each $f$ bounded and measurable (for the Borel sets of $F$) from $F$ to $\R$ (this set is denoted $\mathcal{M}_b(F)$), and for all $x\in F$:
\[
\mathcal{A}f(x) = \sum\limits_{y \in F} \mathcal{A}(x,y)[f(y)-f(x)]
\]
The generator is a bounded linear operator on the bounded functions from $F$ to $\R$. We note in the same way a rate matrix $\mathcal{A} = (\mathcal{A}(x,y))_{x,y \in F}$ and the associated Markov generator $\mathcal{A} : \mathcal{M}_b(F) \to \mathcal{M}(F)$ (with $\mathcal{M}(F)$ the space of measurable functions of $F$). 
 We will always denoting the matrix with two entries (for example $\mathcal{A}(.,.)$) and the generator with no entry (for example $\mathcal{A}$).\\
Conversely giving a generator we can construct the rate matrix looking at the rates (terms that will be before the $[f(y) - f(x)]$) in the expression of the generator.

\begin{definition}\label{Construct Probabiliste Process Markov }

Let $\mathcal{A}$ be a rate matrix and an initial measure $\nu$ we can construct a Markov process $X$ as follow:
\begin{enumerate}
\item Let $(Y(n))_n$ be a Markov chain in $F$ with initial distribution $\nu$ and with transition matrix $\left(\frac{\mathcal{A}(x,y)}  {|\mathcal{A}(x,x)|}\right)_{x,y \in F}$. We allow $|\mathcal{A}(x,x)| = 0$ by taking for one $y_0 \in F$ $\frac{\mathcal{A}(x,y_0)}{|\mathcal{A}(x,x)|} = 1$ and for all $y \neq y_0$ $\frac{\mathcal{A}(x,y)}{|\mathcal{A}(x,x)|} = 0$. We put $y_0$ only to have $(Y_n)_n$ well defined, $y_0$ doesn't have any interest to $X$ as we will see in the following.
\item Let $\Delta_0,\Delta_1,\dots, $ be independent and exponentially distributed with parameter 1 (and independent of $Y(.)$) random variables.
\item We define the Markov process $(X(t))_t$ in $F$ with initial distribution $\nu$ and generator $\mathcal{A}$ by:

\begin{eqnarray}
X(t) = \left\lbrace \begin{array}{ll}
Y(0), & 0 \leq t < \frac{\Delta_0}{|\mathcal{A}(Y(0),Y(0))|} \medskip\\
Y(k), & \sum\limits_{j=0}^{k-1}\frac{\Delta_j}{|\mathcal{A}(Y(j),Y(j))|} \leq t < \sum\limits_{j=0}^{k} \frac{\Delta_j}{|\mathcal{A}(Y(j),Y(j))|}
\end{array}\right.
\end{eqnarray}
\end{enumerate}
Note that we allow $\mathcal{A}(x,x) = 0$ taking $\Delta/0 = \infty$.
\end{definition}
\vspace{15 pt}
Since it will be constant in the evolution, we denote $N$ the (initial) number of particles. Let us firstly introduce the generator of the spatial system (but without birth). Let us begin by some notations \\

\textbf{Notations}\\
	Let $(\Omega,\mathcal{T},\PP)$ be a probability space.
	Let $(e_1,\dots,e_N)$ be the canonical basis of $\R^N$ (with $N\in \N^*$).
	Let $(e_1^1,e_1^2,e_2^1,\dots, e_N^1,e_N^2)$ be the canonical basis of $(\R^2)^N$.\\
	We denote $E = \left((\Z/m\Z)^2 \times \R \times \lbrace 0,1 \rbrace\right)^N$. We identify $\left((\Z/m\Z)^2 \times \R\times \lbrace 0,1 \rbrace\right)^N$ to $((\Z/m\Z)^2)^N \times \R^N \times \lbrace 0,1 \rbrace^N$, then we can write $(x,y,z)\in\left((\Z/m\Z)^2 \times \R \times \lbrace 0,1 \rbrace\right)^N$ with $x\in ((\Z/m\Z)^2)^N,  y\in \R^N,z\in \lbrace0,1\rbrace^N$.\\
    Let us denote ${C}_b(E)$ the set of continuous bounded functions from $E$ to $\R$. \\
    We denote by $\mathfrak{P}_N = \lbrace (i,j) \in \lbrace 1,\dots,N\rbrace / i \neq j \rbrace$ all the individuals couples which can play together.\\

The infinitesimal generator $\mathcal{A}$ with domain $C_b(E)$ of $(X,Y,Z)$ is: for all $f\in C_b(E)$ and for all $(x,y,z) \in E$
\begin{eqnarray}\label{Gene System Spatial}
\mathcal{A} f(x,y,z) = \mathcal{A}_d f(x,y,z) + v\,\mathcal{A}_g f(x,y,z).
\end{eqnarray}

The part $\mathcal{A}_d$ is the generator representing the motion of the individuals. The individuals move following independent random walks, then $\mathcal{A}_d$ is defined by: for all 
 $f\in C_b(E)$: 
\begin{eqnarray}\label{Gene Deplacement}
\mathcal{A}_d f(x,y,z) = \sum\limits_{i=1}^{N} \sum\limits_{c=1}^2 \sum\limits_{\epsilon = \pm 1} \frac{d}{2 \times 2} [f(x + \epsilon e_i^c,y,z) - f(x,y,z)].
\end{eqnarray}
The generator $\mathcal{A}_g$ is the generator representing the evolution of the wealth of individuals through games. Let us give firstly $\mathcal{A}_g$ then the explication of how it works. We have for each function $f\in C_b(E)$: 
\begin{eqnarray}\label{Gene Jeux spatial}
\mathcal{A}_g f(x,y,z) = \hspace{-13pt} \sum\limits_{\substack{(i,j) \in \mathfrak{P}_N}} \frac{1}{2} \mathds{1}_{x_i = x_j} \mathds{1}_{\substack{y_i > 0\\y_j >0}} \left[\begin{array}{l}\ \mathds{1}_{z_i = 0, z_j =0} (f(x,y + \mathbf{R} e_i + \mathbf{R} e_j,z) - f(x,y,z)) + \\
\mathds{1}_{z_i = 1,z_j = 0} (f(x,y + \mathbf{T} e_i - \mathbf{S} e_j,z) - f(x,y,z)) +  \qquad\smallskip \\
\mathds{1}_{z_j = 1,z_i = 0} (f(x,y + \mathbf{T} e_j - \mathbf{S} e_i,z) - f(x,y,z)) +  \qquad\smallskip \\
\ \mathds{1}_{z_i = 1, z_j =1} \left(\begin{array}{l}\ \frac{1}{2}(f(x,y -2\mathbf{P} e_j,z) - f(x,y,z))  \\
+ \frac{1}{2}(f(x,y -2\mathbf{P} e_i,z) - f(x,y,z))\end{array}\right) \ 

\end{array} \right]
\end{eqnarray}
Let us explain a bit this generator:
\begin{itemize}
	\item $\mathds{1}_{x_i = x_j}$ represents the spatial structure meaning that a game makes the wealth change only if the two individuals are on the same site.
	\item $\mathds{1}_{\substack{y_i > 0 \\ y_j >0}}$ checks if both individuals have a positive wealth (\emph{i.e.} there wealth are positive). To have the generator of the ghost system, the only change is: replacing this indicator function by 1.
	\item the big bracket represents the change of wealth through the game of the two individuals. If $z=0$ the individual is a cooperator, if $z=1$ the individual is a defector. The term $\mathds{1}_{z_i=1,z_j=0}$ looks at the strategies of the players and choose the right payoff to give to the individuals. For example if individual 1 and 2 are playing (both have a positive wealth and on the same site) if individual 1 cooperates and 2 defects then individual 1 gets a payoff of $-\mathbf{S}$ and individual 2 gets $\mathbf{T}$, the term in the generator representing this type of transition is: 
	\[
	\mathds{1}_{z_1=0,z_2=1} [f(x,y -\mathbf{S}e_1 + \mathbf{T}e_2,z) - f(x,y,z)].
	\]
\end{itemize}
\begin{Not}
The distribution of the previous $(X,Y,Z)$ defined by the generator (\ref{Gene System Spatial}) is: $\mu^{d,N}$.
\end{Not}

\subsection{Mean Field Model}
\subsubsection{Definition and intuition}
The Mean Field process is a non linear Markov process. That means that the temporal marginals of the distribution $\mu$ of the process is described by a generator also depending on these time marginals $\mu$. For that purpose we call this kind of generator non-linear generator. The change between a non linear Markov process and a homogeneous Markov process is: when determining the transitions in an homogeneous Markov process you have to look where the process is, while in the non linear case to determine the transition you look both where the process is and where it could be.\\

The Mean Field system is the non linear Markov process $(\mathfrak{C}(t),\mathfrak{D}(t))_t$: for $\beta^0 +\rho^0 = 1$ fixed (representing the initial density of cooperators and defectors) defined by: an initial probability measure of $\R_+$, $m_0$ representing the initial distribution of wealth of all individuals and an increment relation:

\begin{eqnarray}\label{eqn Def simpl system}
\mathfrak{D}(t_{n+1}^\mathfrak{D}) - \mathfrak{D}(t_{n}^\mathfrak{D}) = \mathds{1}_{\mathfrak{D(t_n^\mathfrak{D})>0}} \,U_\mathfrak{D}(t^\mathfrak{D}_{n+1})\\
\mathfrak{C}(t_{n+1}^\mathfrak{C}) - \mathfrak{C}(t_{n}^\mathfrak{C}) = \mathds{1}_{\mathfrak{C(t_n^\mathfrak{C})>0}} \,U_\mathfrak{C}(t^\mathfrak{C}_{n+1})
\end{eqnarray}
where:
\begin{itemize}
	\item[-] $(t_n^\mathfrak{D})_n$ and $(t_n^\mathfrak{C})_n$ are sequences of Poisson times of intensity ${v}$.
	\item[-] for $I = \lbrace t_n^\mathfrak{C} ,n \in \N \rbrace \cup \lbrace t_n^\mathfrak{D},n\in \N \rbrace$, all $t \in I$ and $U_\mathfrak{D}(t)$ and $U_\mathfrak{C}(t)$ are random variables independent from everything such that: 
	\[
	U_\mathfrak{D}(t) = \left\{\begin{array}{rl}
	 - 2 \mathbf{P} & \text{with probability } \frac{1}{2}\rho^0\PP(\mathfrak{D}(t)>0)\smallskip\\
	 \mathbf{T} & \text{with probability } \beta^0\PP(\mathfrak{C}(t)>0)\smallskip\\
	 0 & \text{ with probability } 1 - \beta^0\PP(\mathfrak{C}(t)>0) - \frac{1}{2}\rho^0\PP(\mathfrak{D}(t)>0)
	\end{array}\right.
	\]
 and 
	\[
	U_\mathfrak{C}(t) = \left \lbrace\begin{array}{rl}
	 -  \mathbf{S} & \text{with probability } \rho^0\PP(\mathfrak{D}(t)>0)\smallskip\\
	 \mathbf{R} & \text{with probability } \beta^0\PP(\mathfrak{C}(t)>0)\smallskip\\
	 0 & \text{ with probability } 1 - \beta^0\PP(\mathfrak{C}(t)>0) - \rho^0\PP(\mathfrak{D}(t)>0)
	\end{array}\right.
	\]
\end{itemize}

$(\mathfrak{C}(t),\mathfrak{D}(t))_t$ represents the wealth of a typical cooperator and a typical defector in an infinite population of individuals moving at a high speed.\\

\subsubsection{Justification of the mean field model} \label{Subsect Justif}

To arrive to the Mean Field model from the spatial model we have to assume that:
\begin{itemize}
\item There is no birth anymore. As a consequence $Z$ is constant over time. We also set $N \in \N$ the number of particles in the system.
\item The collection $(Z^i)_{1\leq i\leq N}$ are i.i.d with distribution $\beta^0 \delta_0 + \rho^0 \delta_1$ (with $\beta^0 + \rho^0 = 1$ previously fixed).
\item The moving rate $d$ belongs to $\N$ instead of $\R_+$. This is done in order to have an homogenization result.
\item The size of the torus depends on the initial number of particles. We set: $m^2 = N$. (We can also not link the size of the torus and the initial number of particles but in that case we have to slow the time of evolution replacing $(X(t),Y(t),Z)_t$ by $(X(t/N),Y(t/N),Z)_t$).
 \item Initially the distribution of wealth of all individuals is $m_0$ probability measure of $\R_+$. 
\end{itemize}
The generator of the spatial system is $\mathcal{A}$ with domain $C_b(E)$: for all $(x,y,z) \in E$
\[
\mathcal{A} f(x,y,z) = \mathcal{A}_d f(x,y,z) + v\,\mathcal{A}_g f(x,y,z).
\]
with $\mathcal{A}_d$ defined in (\ref{Gene Deplacement}) and $\mathcal{A}_g$ in (\ref{Gene Jeux spatial}). \\
Using Theorem 4.2 of \cite{gibaud2016spatialized}, when $d$ goes to $+\infty$, we have that there exists $\mu$ probability measure of $D(\R_+,\R\times \lbrace0,1 \rbrace)$ such that the particle system described by the generator (\ref{Gene System Spatial}) is $\mu$-chaotic. The analytic way of describing $\mu$ is done in the following. The useful consequence of that is the following convergence in finite dimensional distribution: with $(Y,z)$ random variable taking value in $D(\R_+,\R\times \lbrace0,1 \rbrace)$ and with distribution $\mu$ 

\begin{eqnarray}\label{eqnarray Consequence Convergence MF}
\forall i\in\N, \quad (Y^i,Z^i) \underset{\overset{\longrightarrow}{\substack{d \to +\infty \\ \text{then}\\N \to +\infty}}}{\mathcal{D}} (Y,z).
\end{eqnarray}
Let us describe analytically $\mu$. In order to make the notations lighter we use the following: \begin{itemize}
	\item  $\rho_t:= \PP(Y(t)>0|z=1)$,
	\item  $\beta_t:= \PP(Y(t)>0|z=0)$.
\end{itemize}
To have a good description of the distribution of $(Y(t),z)_t$ we use evolution equations given in Corollary 4.3 of \cite{gibaud2016spatialized}. 
The distribution of $(Y(t),z)_t$ verifies the following evolution equations: with $\mathcal{L}(Y(0)) = m_0$ and $\mathcal{L}(z)= \rho \delta_0 + (1 - \rho)\delta_1$:
\begin{itemize}
	\item for all $y \in \R$
	\begin{eqnarray} \label{Evol B MF}
	\begin{array}{rcl}
	\dro \PP(Y(t) = y|z = 0) & = &{v}\mathds{1}_{y  > \mathbf{R}} \beta^0\beta_t\PP(Y(t) = y - \mathbf{R} | z =0)  \smallskip\\
	& & + {v}\rho^0\rho_t \PP(Y(t) = y + \mathbf{S}|z=0) \smallskip\\
	& & - {v}(\rho^0\rho_t + \beta^0\beta_t) \PP(Y(t)=y|z =0)
	\end{array}
	\end{eqnarray}
	
	\item for all $y \in \R$
	
	\begin{eqnarray} \label{Evol R MF}
	\begin{array}{rcl}
	\dro \PP(Y(t) = y|z = 1) & = &{v}\mathds{1}_{y > \mathbf{T}}\beta^0\beta_t\PP(Y(t) = y - \mathbf{T}|z = 1)  \smallskip\\
	& & + {v}\frac{\rho^0\rho_t}{2}\PP(Y(t) = y + 2\mathbf{P}|z = 1) \smallskip\\
	& & + {v}\frac{\rho^0\rho_t}{2}\PP(Y(t) = y|z = 1) \smallskip\\
	& & - {v}(\rho^0\rho_t + \beta^0\beta_t) \PP(Y(t)=y|z=1).
	\end{array}
	\end{eqnarray}	
\end{itemize}
\vspace{10pt}
We want to prove that $(\mathfrak{C},\mathfrak{D})$ is a good description of $(Y(t),z)_t$. Firstly let us denote $\beta'_t = \PP(\mathfrak{C}(t)>0)$ and $\rho'_t = \PP(\mathfrak{D}(t)>0)$.
Let's verify that $\left(\mathfrak{C}(t)\right)_t$ and $\left(\mathfrak{D}(t)\right)_t$ verifies the ordinary differential equation (\ref{Evol B MF}),(\ref{Evol R MF}).
Because for all $y \in \R_+$ $ \forall y' \notin  \lbrace y + x : x \in \{-\mathbf{R},\mathbf{S}\} \rbrace$, ${\PP(\mathfrak{C}({t+h}) = y | \mathfrak{C}(t) = y') = o(h)}$ we have:
\begin{align*}
\PP(\mathfrak{C}(t+h) = y) & = \PP\left(\mathfrak{C}(t+h) = y  \cap \mathfrak{C}(t) = y\right)+ \PP\left(\mathfrak{C}(t+h) = y \cap \mathfrak{C}(t) \neq y\right) \medskip\\
 & = \PP\left(\mathfrak{C}(t+h) = y | \mathfrak{C}(t) = y\right) \PP\left(\mathfrak{C}(t) = y\right) \\
  & \hspace{10pt} + \PP\left(\mathfrak{C}(t+h) = y \cap \mathfrak{C}(t) \in \lbrace y - \mathbf{R},y + \mathbf{S}\rbrace\right) + o(h) \medskip\\
  & = \left(1 - v h (\beta^0 \beta'_t + \rho^0 \rho'_t) \right) \PP\left(\mathfrak{C}(t) = y\right) \\
   & \hspace{10pt}+ v h \beta^0 \beta'_t \mathds{1}_{y > \mathbf{R}} \PP\left(\mathfrak{C}(t) = y - \mathbf{R}\right) + vh \rho^0 \rho'_t \PP\left(\mathfrak{C}(t) = y + \mathbf{T}\right) + o(h).
\end{align*}

 $\PP(\mathfrak{C}(t+h) = y)$ is equal to:
\[
\PP(\mathfrak{C}(t) = y){v h}(1 - \rho_t - \beta_t) + \mathds{1}_{k>\mathbf{R}}{v h}\beta_t\PP(\mathfrak{C}(t) = y - \mathbf{R}) + {v h}\rho_t \PP(\mathfrak{C}(t) = y + \mathbf{S})+o(h).
\]
Dividing by $h$ and making $h\to 0$ we get the evolution equations (\ref{Evol B MF}) and (\ref{Evol R MF}) by saying that for all $y\in \R$ $\PP(Y(t)=y|z = 0) = \PP(\mathfrak{C}(t) = y)$ and $\PP(Y(t) = y|z=1) = \PP(\mathfrak{D}(t) =y)$ .\\

Since this process is non linear, we analyze a linearized version $(\overline{\mathfrak{C}}(t),\overline{\mathfrak{D}}(t))_t$. The linearized system is a system where cooperators and defectors can play even if their wealth are negative. It is a Mean field ghost version. The wealth of cooperators in the linearized $(\overline{\mathfrak{C}}(t))_t$ is defined by an initial measure $m_0$ and an increment relation:
	  	 \begin{eqnarray}
	  	 \overline{\mathfrak{C}}(t_{n+1}^\mathfrak{C}) - \overline{\mathfrak{C}}(t_{n}^\mathfrak{C}) = \overline{U}_\mathfrak{C}(t^\mathfrak{C}_{n+1})
	  	 \end{eqnarray}
	  	 where:
	  	 \begin{itemize}
	  	 	\item[-] $(t_n^\mathfrak{C})_n$ are defined in (\ref{eqn Def simpl system}).
	  	 	\item[-] for all $t \in \lbrace t_n^\mathfrak{C} ,n \in \N \rbrace$ and  $\overline{U}_\mathfrak{C}(t)$ are random variables independent from everything such that: 
	   	 	\[
	  	 	\overline{U}_\mathfrak{C}(t) = \left \lbrace\begin{array}{rl}
	  	 	 -  \mathbf{S} & \text{with probability } \rho^0\smallskip\\
	  	 	 \mathbf{R} & \text{with probability } \beta^0\smallskip\\
	  	 	 0 & \text{ with probability } 1 - \beta^0 - \rho^0
	  	 	\end{array}\right.
	  	 	\]
	  	 \end{itemize}

The consequence of this linearization is the following theorem that gives us a simple condition to have "survival" of cooperators \emph{ad vitam eternam} with positive probability.
\begin{thm}\label{thm Modele simplifié Garanti}
	Let us suppose that:
	
	\begin{eqnarray}\label{Hyp Drift positif}
	\beta^0 \mathbf{R} - \rho^0 \mathbf{S} >0
	\end{eqnarray}
	
	Then we have: for a fixed $q_0 >0$ {when } $t\to +\infty$ $\PP_{q_0}$ {almost surely}
	\[
	\overline{\mathfrak{C}}(t) \longrightarrow +\infty \qquad 
	\]
	and
	\[
	\PP_{q_0} (\forall t\geq 0, \overline{\mathfrak{C}}(t)>0) >0.
	\]
\end{thm}

We also have a result about the concentration of wealth of cooperators.

\begin{prop}\label{Prop Concentration Wealth MF DPD}
For all $q_0 \in \R_+$

We denote :\[
\mathfrak{m} := {v}\left(\beta^0 \mathbf{R} - \rho^0\mathbf{S}\right), \qquad \sigma^2 := {v}\left(\beta^0 \mathbf{R}^2 + \rho^0 \mathbf{S}^2\right).
\]

Then we have : $\forall \eta>0,t\geq0$
\begin{eqnarray}\label{eqn tchebychev MF DPD}
\mathbb{P}_{q_0}\left( \overline{\mathfrak{C}}(t)\in \left[ q_{0}+\mathfrak{m}t-\eta \sqrt{\sigma ^{2}t},q_{0}+\mathfrak{m}t+\eta \sqrt{\sigma
^{2}t}\right] \right) \geq 1-\eta ^{-2}.
\end{eqnarray}

Moreover for all $\tau$ such that:

\[\tau <q_{0}-\frac{\eta ^{2}\sigma ^{2}}{4\mathfrak{m}},\]

then 

\[\mathbb{P}(\overline{\mathfrak{C}}(t)>\tau )\geq 1-\eta ^{-2}.\]

\end{prop}

\vspace{10pt}

\section{Proofs on the Spatial Model}
Before beginning the proofs let us introduce some useful tools. 
In Section \ref{Spatial Model}, every movement, game, birth is managed by independent Poisson processes with rates $d$,$v$ or ${b}$. Instead of considering this collection of Poisson processes we consider: \begin{enumerate}
	\item A unique Poisson process with rate $K(b+d) + \frac{K(K-1)}{2} v$ \emph{i.e.} the sum of the rates of the previous Poisson processes. This Poisson process gives a sequence of Poisson times $(t_n)_n$
	\item A sequence (independent of everything) of independent random variables $(C_n)_n$ selects if a movement, game or birth happens. An element of this sequence choose a movement with probability $ \frac{K d}{K(b+d) + \frac{K(K-1)}{2} v}$, a game with probability $ \frac{\frac{K(K-1)}{2} v}{K(b+d) + \frac{K(K-1)}{2} v}$ and a birth with probability $ \frac{K b}{K(b+d) + \frac{K(K-1)}{2} v}$.
	\item 3 sequences $(E^m_n)_n, (E^g_n)_n, (E^b_n)_n,$ (independent of everything) of i.i.d random variables uniformly choosing the agents. $(E^m_n)_n$ and $(E^b_n)_n$ are made of uniform on $\lbrace 1,\dots K \rbrace$ random variables. $(E^g_n)_n$ is made of uniform on $\left\lbrace 1,\dots, \frac{K(K-1)}{2} \right\rbrace$ random variables.
	\item At the $n$-th realization of the unique Poisson process, $C_n$ chooses the type of evolution (movement, game, or birth). If $C_n = $ movement (resp birth) then $E^m_n$ (resp $E^b_n$) gives the particle that will move (resp will give birth (if she is wealthy enough and if there are less than $K$ particles on the torus)). If the particle moves, we draw $V_n$ an other uniform on $\lbrace Top, Bottom, Left,Right\rbrace$ random variable to decide where the particle moves. If $C_n=$ game then $E^g_n$ gives the couple of particles that will play (if their $Z \neq -1$, if they are on the same site and if their wealths are positive).
	\end{enumerate}
	In this section, instead of considering continuous time Markov process $(X(t),Y(t),Z(t))_t$ we consider the induced Markov chain $(X_n,Y_n,Z_n)_n:= (\sigma_n)_n$ where $\forall n \in \N$ $(X_n,Y_n,Z_n) = (X({t_n}),Y(t_n),Z(t_n))$.\\

\begin{definition}
We call the photograph of a configuration $\sigma$, the data of the positions of individuals and their strategies. The photograph of $\sigma$ can be seen as the equivalent class of $\sigma$ for the following equivalence relation: $\sigma \sim \tilde{\sigma}$ if and only if $\forall x\in (\Z/m\Z)^2,\, \forall z\in \lbrace 0,1 \rbrace$ \[\text{Card} \left\lbrace i \in \lbrace1,\dots K \rbrace,\  X^i =x,Y^i>0,Z^i=z \right\rbrace = \text{Card} \left\lbrace i \in \lbrace1,\dots K \rbrace, \  \tilde{X}^i =x,\tilde{Y}^i>0,\tilde{Z}^i=z \right\rbrace.
\]
In words, two configurations are equivalent if and only if on each position $x\in (\Z/m\Z)^2$ there is the same number of cooperators with positive wealth and defectors with positive wealth.\\
We denote by $\mathfrak{p}(\bullet)$ the canonical projection on the space of photographs. 
\end{definition}
\begin{remarque}
There are less than $\prod\limits_{N=0}^{K} (2m^2)^N = (2m^2)^{\frac{K(K+1)}{2}}$ photographs. We call $q$ the number of photographs. Those photographs are denoted $\mathfrak{p}_1,\dots, \mathfrak{p}_{q}$.
\end{remarque}
\begin{Not}
For simplicity, we will denote the wealth of cooperator $i$ (with $Z^i = 0$), $\mathfrak{C}^i$ and the wealth of defector $j$ (with $Z^j = 1$), $\mathfrak{D}_j$.\\
For a configuration $\sigma$ we denote $N_\mathfrak{C}(\sigma)$ (resp. $\N_\mathfrak{D}(\sigma)$) the set of the indexes of the cooperators (resp. defectors) with positive wealth.
\end{Not}

\subsection{Almost sure extinction}
In this section we prove Theorem \ref{Intro thm Extinction presque sure}.

\begin{proof} \textcolor{white}{text}\\

	Firstly let us define the first cooperator-defector time. For that let us denote for all $n\in \N$, $\mathfrak{C}^{tot}_n:= \sum_{i = 1}^{K} \mathds{1}_{Z^i = 0} Y^i_n$ the sum of wealth of cooperators. 
		\begin{definition}
			We define $\tau$ the first time when a cooperator-defector game happens that is:
			\[
			\tau_1 = \inf \left\lbrace n\in \N^* / \mathfrak{C}^{tot}_{n} - \mathfrak{C}^{tot}_{n-1} < 0 \right\rbrace
			\]
			Let us denote $(\tau_i)_i$ the sequence of these cumulated stopping time with $\tau_0 = 0$. If there is no cooperator left after $\tau_i$ we define for all $k\in \N^*$ $\tau_{i+k} = +\infty$ a.s. . That is the time of the first Cooperator-Defector game is $\tau_1$, the time of the second one is $\tau_2$ et so on. 
		\end{definition}
		Firstly let us prove the uniform (on the configurations) upper bounding lemma. 
		\begin{lemme}\label{Lemme Majoration uniforme Extinction}
			There is $\mathfrak{m}\in \N$ and $\varepsilon >0$ such that for each configuration $\sigma \in \mathcal{C}'$  and all $s\in \R$ we have: 
			\[
			\PP_\sigma(\tau_1 > s) \leq \left( 1 - \varepsilon\right)^{\lfloor s / \mathfrak{m} \rfloor}
			\]
		\end{lemme}
		\begin{proof}

			Firstly let us notice, if there is at least one defector initially, defectors cannot be extinct. Indeed the only way that a defector dies is via Defector-Defector game yet since the payoff of such a game is a random variable with distribution $\frac{1}{2}\delta_{(0,-2\mathbf{P})} + \frac{1}{2} \delta_{(-2\mathbf{P},0)}$ then, when they play together two defectors cannot die at the same time. As a consequence there will be at least one defector left at any time. Moreover since $(\Z/m\Z)^2$ is finite we have for each $\sigma \in \mathcal{C}'$ $\tau_1 < +\infty$ $\PP_\sigma$ a.s.\\
			
			Let $(T_n)_n$ be the sequence of events $T_n = \lbrace \tau_1 = n\rbrace$.
			Let $\mathfrak{m}$ the maximum number of realizations of Poisson process for a defector to play with a cooperator with positive probability in all configurations \emph{i.e.} 
			\[
			\mathfrak{m} = \max\limits_{\sigma \in \mathcal{C}'} \min \lbrace n \in \N / \PP_\sigma (T_n) >0 \rbrace.
			\]
			Because there is a finite number of photographs the existence of $\mathfrak{m}$ is insured. \\
			Then there is $\varepsilon>0$ such that for each configuration:
			
			\begin{eqnarray}\label{Inegal majoration uniforme en configuration Proba}
			\PP_\sigma(\tau_1 \leq \mathfrak{m}) \geq \varepsilon
			\end{eqnarray}

			For each configuration $\sigma^0 \in \mathcal{C}'$ and all $k\in \N^*$ we have:
			\begin{align*}
			\PP_{\sigma^0}(\tau_1>k\mathfrak{m}) &= \PP_{\sigma^0}(\tau_1 > \mathfrak{m})\,\PP_{\sigma^0}(\tau_1>k\mathfrak{m}|\tau_1 > \mathfrak{m})\\
			 & \leq (1 - \varepsilon) \sum_{\sigma^1 \in \mathcal{C}'} \PP_{\sigma^0}(\tau_1>k\mathfrak{m},\sigma_\mathfrak{m}=\sigma^1|\tau_1 > \mathfrak{m})
			\end{align*}
			Moreover by the Strong Markov property, we have:
			\begin{align*}
			\PP_{\sigma^0}(\tau_1>k\mathfrak{m})	& \leq (1- \varepsilon)\sum_{\sigma^1 \in \mathcal{C}'} \PP_{\sigma^0} (\tau_1 > k \mathfrak{m} | \sigma_\mathfrak{m} = \sigma^1, \tau_1 > \mathfrak{m}) \PP_{\sigma^0}(\sigma_\mathfrak{m} = \sigma^1|\tau_1 > \mathfrak{m}) \medskip\\
	 &		\leq (1 - \varepsilon)\sum\limits_{\sigma^1 \in \mathcal{C}'} \PP_{\sigma^1}(\tau_1 > (k-1)\mathfrak{m})\PP_{\sigma^0}(\sigma_\mathfrak{m} = \sigma^1 | \tau_1 > \mathfrak{m}) \medskip\\
					& \leq (1 - \varepsilon)^{k}\prod_{i = 1}^{k} \underbrace{\sum_{\sigma^i \in \mathcal{C}'} \PP_{\sigma^{i-1}}(\sigma_{\mathfrak{m}} = \sigma^i| \tau_1 > \mathfrak{m})}_{=1}\medskip\\
					& \leq (1 - \varepsilon)^k\\
					\end{align*}
	\end{proof}
	Now that the lemma is proved let us use it to prove the theorem.\\ 
	\begin{align*}
	\E_\sigma({\mathfrak{C}}^{tot}_{\tau_{n+1}} - {\mathfrak{C}}^{tot}_{\tau_{n}} | \exists i \in N_\mathfrak{C}(\tau_n)\  {\mathfrak{C}}^{i}_{\tau_{n}} > 0) &= \sum\limits_{\sigma' \in \mathcal{C}'} \E_\sigma \left(\left({\mathfrak{C}}^{tot}_{\tau_{n+1}} - {\mathfrak{C}}^{tot}_{\tau_{n}} \right) \mathds{1}_{{\sigma}_{\tau_n}} = \sigma' | \exists i \in N_\mathfrak{C}(\tau_n) \ {\mathfrak{C}}^{i}_{\tau_{n}} > 0\right)\\
	& = \sum\limits_{\sigma' \in \mathcal{C}'} \E_{\sigma'}\left({\mathfrak{C}}^{tot}_{\tau_{1}} - {\mathfrak{C}}^{tot}_{0} \right) \PP_\sigma\left({{\sigma}_{\tau_n} = \sigma' | \exists i \in N_\mathfrak{C}(\tau_n)\  {\mathfrak{C}}^{i}_{\tau_{n}} > 0}\right)\\
	& \leq \sum\limits_{\sigma' \in \mathcal{C}'} \left(\E_{\sigma'} \left({\tau}_1\right) \mathbf{R} - \mathbf{S}\right) \PP_\sigma({\sigma}_{\tau_n} = \sigma'| \exists i \in N_\mathfrak{C}(\tau_n) \ {\mathfrak{C}}^{i}_{\tau_{n}} > 0) \\
	& \leq \mu \mathbf{R} - \mathbf{S}.
		\end{align*}
		
		The hypothesis of the theorem is: 
		\[
		\nu = \mu \mathbf{R} - \mathbf{S} < 0
		\]

	We have: for all $n\in \N$
	\[
	\mathfrak{C}^{tot}_{\tau_{n+1}} = \sum_{k=0}^{n} \mathfrak{C}^{tot}_{\tau_{k+1}} - \mathfrak{C}^{tot}_{\tau_{k}} + \mathfrak{C}^{tot}_{0}
	\]
	Let $\sigma \in \mathcal{C}'$ we get:
	\begin{align*}
	\E_\sigma (\mathfrak{C}^{tot}_{\tau_{n+1}})& = \E_\sigma(\mathfrak{C}^{tot}_{0}) + \sum_{k=0}^{n} \E_\sigma(\mathfrak{C}^{tot}_{\tau_{k+1}} - \mathfrak{C}^{tot}_{\tau_{k}}) \\
	 &= \E_\sigma(\mathfrak{C}^{tot}_{0}) + \sum_{k=0}^{n} \E_\sigma(\mathfrak{C}^{tot}_{\tau_{k+1}} - \mathfrak{C}^{tot}_{\tau_{k}}| \exists i\in N_\mathfrak{C}(\tau_k)\  \mathfrak{C}^i_{\tau_{k}} > 0) \PP_\sigma(\exists i \in N_\mathfrak{C}(\tau_k)\  \mathfrak{C}^i_{\tau_{k}} > 0) \\ & \hspace{13 pt}+  \sum_{k=0}^{n} \E_\sigma(\mathfrak{C}^{tot}_{\tau_{k+1}} - \mathfrak{C}^{tot}_{\tau_{k}}| \forall i \in N_\mathfrak{C}(\tau_k) \ \mathfrak{C}^i_{\tau_k} \leq 0) \PP_\sigma (\forall i \in N_\mathfrak{C}(\tau_k) \  \mathfrak{C}^i_{\tau_k} \leq 0)\\
	 & \leq \E_\sigma(\mathfrak{C}^i_{0}) + \nu \sum_{k = 0}^{n} \PP_\sigma(\exists i \in N_\mathfrak{C}(\tau_k)\  \mathfrak{C}^i_{\tau_{k}} > 0)
	\end{align*}
	But $\nu <0$ and $\forall n \in \N$ $\E_\sigma(\mathfrak{C}^{tot}_{\tau_{n}}) > - K \mathbf{S}$ then we have:
	\[
	\sum_{k =0}^{+\infty} \PP_\sigma(\exists i \in N_\mathfrak{C}(\tau_k)\ \mathfrak{C}^i_{\tau_{k}} > 0) < +\infty.
	\]
	Applying Borel Cantelli's Lemma we have: $\PP_\sigma$ almost surely eventually all cooperators will be dead.
	 
\end{proof}

\subsection{Coexistence \emph{ad vitam eternam}}
Before proving Theorem \ref{thm Coexistence ad vitam eternam Spatial}, let us prove an almost sure divergence lemma.
\begin{lemme}
Let $(X_n)_n$ a Markov chain on $\R$ such that:\begin{itemize}
\item $\exists X$ random variable in $\R_+$ such that: $\forall x_0 \in \R$, $\E_{x_0} (X) \leq \kappa$, $\E_{x_0} (X^2) \leq \kappa'$ and $\E_{x_0} (X^4) \leq \kappa''$ and such that: $\forall x_0 \in \R$ $\forall n \in \N$
\[
|X_{n+1} - X_n| \leq X \quad \PP_{x_0}\  a.s.
\]

\item $\exists$ $\tilde{X}$ random variable from $(\Omega,\mathcal{A},\PP)$ to $\R$  with positive expectation $\E(X)=\alpha >0$ such that:
\[
\forall s\in\R,\ \forall x_0\in \R, \qquad \PP_{x_0}(X_1 - X_0 > s) \leq \PP(\tilde{X}>s)
\]
\end{itemize} 
Then we have: for a fixed $0<\delta< \alpha$ $\PP_\sigma$ almost surely: $\exists N_0(\sigma) \in \N$, $\forall N> N_0(\sigma)$
		
		\[
		X_N  > \delta N \longrightarrow +\infty.
		\]
\end{lemme}

\begin{cort}\label{Corol Divergence a.s.}
Let $(\sigma_n)_n$ a Markov chain of a space $E$ with countable dimension. Let $(X_n)_n$ be one of its coordinates. Let us suppose  \begin{itemize}
\item $\exists X$ random variable in $\R_+$ such that: $\forall x_0 \in \R$, $\E_{x_0} (X) \leq \kappa$, $\E_{x_0} (X^2) \leq \kappa'$ and $\E_{x_0} (X^4) \leq \kappa''$ and such that: $\forall x_0 \in \R$ $\forall n \in \N$
\[
|X_{n+1} - X_n| \leq X \quad \PP_{x_0}\  a.s.
\]

\item $\exists$ $\tilde{X}$ random variable from $(\Omega,\mathcal{A},\PP)$ to $\R$  with positive expectation $\E(X)=\alpha >0$ such that:
\[
\forall s\in\R,\ \forall x_0\in \R, \qquad \PP_{x_0}(X_1 - X_0 > s) \leq \PP(\tilde{X}>s)
\]
\end{itemize} 
Then we have: for a fixed $0<\delta< \alpha$ $\PP_\sigma$ almost surely: $\exists N_0(\sigma) \in \N$, $\forall N> N_0(\sigma)$
		
		\[
		X_N > \delta N \longrightarrow +\infty.
		\]
\end{cort}
The proof of the corollary is the same as the proof of the lemma except that the initial condition is for example with $\sigma \in E$ $\lbrace\sigma_0 = \sigma\rbrace$ and the Strong Markov property are done on $(\sigma_n)_n$. Let us begin the proof of the lemma.
\begin{proof}
To prove that $X_n \longrightarrow + \infty$ almost surely we firstly prove that $(X_n)_n$ is a sub-martingale and use Doob decomposition, then we upper bound the expectation of the quadratic variation of the martingale part of $X_n$ in order to finally use Borel Cantelli's Lemma with Burkholder-Davis-Gundy inequality.
	\begin{enumerate}
	\item Let us prove that $(X_n)_n$ is a sub-martingale. Let $x\in \R$.
	\begin{align*}
		\E_x(X_{n+1} | X_n) & = \E_x \left(X_n + X_{n+1} - X_{n} | X_n\right) \\ 
		 & = X_{n} + \E_{X_n}  \left(X_1 - X_{0}\right) \\
		 & \geq X_{n} + \alpha \\
		 & \geq X_{n}.
		\end{align*}
		Then using a Doob decomposition we get: for all $x\in \R$
		 \[X_{n} = M_n + \sum\limits_{k = 0}^{n-1} \E_x \left(X_{k+1} - X_{k} \vert X_k \right) \] with $M_0 =0$ and $(M_n)_n = \left(\sum\limits_{k=0}^{n-1} X_{k+1} - \E_x\left(X_{k+1} | X_{k}\right)\right)_n$ a martingale.
		
		\item Now let us upper bound for all $x\in \R$ and all $N \in \N$, $\E_x\left([M_N]^2\right) = \E_x \left(\left(\sum\limits_{k = 0}^{N-1} (M_{k+1} - M_{k})^2\right)^2\right)$
		\begin{itemize}
		\item Let us prove that $\exists c>0$ such that $\forall k \in \N$ $\forall x \in \R$ $\E_\sigma((M_{k+1} - M_{k})^2) \leq c$.
		First let us rewrite $M_{k+1} - M_k$ in function of the variation of $X_{k}$.
		\[
		M_{k+1} -M_k = X_{k+1} - X_{k} - \E_x(X_{k+1} - X_{k} | X_{k}).
		\]
		We get using Strong Markov property:
		\[
		\E_x\left(\vert X_{k+1} - X_{k}\vert \, | X_k\right) = \E_{X_k} (|X_{1} - X_{0}|) \in \left[- \kappa,\kappa \right]
 		\]
 		Then we get:
 		\[
 		\left|X_{k+1} - X_{k} - \E_x(X_{k+1} - X_{k} |X_k)\right| \leq \left|  X_{k+1} - X_{k} \right| + \kappa
 		\]
 		Then using Strong Markov property we have: $\forall x \in \R$
	 	
	 	\[\E_\sigma((M_{k+1} - M_k)^2) \leq  \E_\sigma\left(\E_{X_k}\left(X_{1} - X_{0}\right)^2\right)  + 3 \kappa^2 \leq \kappa' + 3 \kappa^2 := c\]

	 	\item Let us prove that $\exists c'>0$ such that $ \forall N\in \N$ and $\forall x \in \R$ $\E_x (\left[M_N\right]^2) \leq N^2 c^2 + N c'$\\
	 		Using the previous arguments we have: $\exists c' >0$ such that for all $k \in \N$:\[
	 		 	 \E_x((M_{k+1} - M_k)^4) \leq c'
	 		 	\]
	 	
	 	We have: 
	 	\begin{align*}
	 	\E_x (\left[M_N\right]^2) &= \sum_{0 \leq k,\ell \leq N - 1} \E_x\left((M_{k+1} - M_k)^2 (M_{\ell +1} - M_\ell)^2\right) \\  &=\sum_{0 \leq k  \neq \ell \leq N - 1} \E_x\left((M_{k+1} - M_k)^2 (M_{\ell +1} - M_\ell)^2\right) + \sum_{k = 0}^{N- 1}\E_x\left( (M_{k + 1} - M_k)^4\right)
	 	\end{align*}

	 	 Using Strong Markov property we have for all $k < \ell$:
	 	 \begin{align*}
	 	 \E_x \left((M_{k+1} - M_k)^2 (M_{\ell + 1} - M_\ell)^2\right) & = \E_x \left((M_{k+1} - M_k)^2 \E_x\left((M_{\ell + 1} - M_\ell)^2| \mathcal{F}_{\tau^i_\ell}\right)\right) \\
	 	 &= \E_x\left((M_{k+1} - M_k)^2 \E_{X_{\tau^i_\ell}} (M_{\ell - k + 1} - M_{\ell - k})  \right)\\
	 	 &\leq c \E_x((M_{k+1} - M_k)^2) \leq c^2.
	 	 \end{align*}
	 	 Finally we have:
	 	 \begin{eqnarray}
	 	 \E_x (\left[M_N\right]^2) \leq N^2 c^2 + N c'
	 	 \end{eqnarray}
		\end{itemize}
		\item Firstly let us notice that $\E_x (\left[- M_N \right]^2) \leq N^2 c^2 + N c'.$ \\
		To finish the divergence of $(X_n)_n$ we need Burkhholder-Davis-Gundy inequality \cite{beiglbock2015pathwise} which says that for $1\leq p< +\infty$ there exists $a_p < +\infty$ such that for every $N \in \N$ and every $(\mathcal{M}_n)_n$ martingale we have:
		\[
		\E\left( \left(\max_{1\leq n\leq N}\mathcal{M}_n\right)^p\right) \leq a_p \E \left(\left[ \mathcal{M}_N\right]^{p/2}\right)
		\]
		Then we have:  $0<\delta< \alpha$
		\begin{align*}
		\PP_x\left(\min_{0 \leq n\leq N} M_n < -N (\alpha + \delta) \right) &= \PP_x\left(\max_{0 \leq n \leq N} -M_n > N(\alpha + \delta)\right) \\
		& \leq \PP_x \left( \left(\max_{0 \leq n \leq N} -M_n\right)^4 > {N}^4(\alpha + \delta)^4 \right) \\
		& \leq \frac{1}{N^4 (\alpha + \delta)^4 } \E_x\left(\left(\max_{0 \leq n \leq N} -M_n\right)^4\right) \\
		& \leq a_4 \frac{1}{N^4 (\alpha + \delta)^4 } \E_x (\left[-M_N\right]^2) \leq a_4 \frac{N^2c^2 + Nc'}{N^4 (\alpha + \delta)^4 }  
		\end{align*}
		which is summable. \\
		Applying Borel Cantelli's lemma we have $\PP_x$ almost surely: $\exists N_0(\sigma) \in \N$, $\forall N> N_0(x)$
		
		\[
		X_N \geq M_N + \alpha N > \delta N \longrightarrow +\infty.
		\]
		\end{enumerate}  
\end{proof}
Let us now prove Theorem \ref{thm Coexistence ad vitam eternam Spatial}.\\
\begin{proof}
	The configuration $\sigma^0$ is fixed for the sequence of the proof.

	\begin{definition}
		For all configuration $\sigma \in \mathcal{C}$, for a cooperator $i \in {N}_\mathfrak{C}(\sigma)$  we denote ${\tau}_1^i$ the first time where there is a game between $i$ and a cooperator.
		\[
		{\tau}_1^i = \inf \left\lbrace n\in \N^* / {\mathfrak{C}}^i_{n} - {\mathfrak{C}}^i_{n-1} > 0 \right\rbrace
		\]
		Let us denote $({\tau}^i_n)_n$ the sequence of these cumulated stopping times (by the same way we define the cumulative stopping times in the previous proof) with ${\tau}^i_0 = 0$ for all cooperators.
	\end{definition}

	The issue here is that we can have extinction of the cooperators. Hence we don't have $\forall \sigma \in \mathcal{C}$ $\forall i \in N_\mathfrak{C}(\sigma)$ ${\tau}_1^i < +\infty$ $\PP_\sigma$ a.s.\\
	Hence we have to consider an other system and couple them. \begin{enumerate}
		\item We call the system introduced in Section \ref{Spatial Model} the \textbf{True System}. Let $\mathfrak{C}^i_n$ be the wealth of cooperator $i \in N_\mathfrak{C}(\sigma) $ at time $n\in \N$ in the \textbf{True System}.
		\item The other system is called the \textbf{Ghost system}. In this system we have that: \begin{itemize}
			\item cooperators and defectors can play even if they have negative wealth,
			\item cooperators cannot give birth, 
			\item at each step of time, there is a decrease of $w_0$ of cooperator wealths,
			\item when defectors give birth they don't loose wealth (this part is not necessary to make the proof but make the proof easier).
			\end{itemize}
	\end{enumerate}

	We denote the probability measure associated to the \textbf{Ghost system} $\G$. For example probability that the wealth of cooperator $i \in {N}_\mathfrak{C}(\sigma^0)$ is 3 in the \textbf{Ghost system} is denoted $\G_{\sigma^0}\left({\mathfrak{C}}^i_n = 3\right)$.\\
	It is the following system $(\overline{\sigma}_t)_t$ that we couple to $(\sigma_t)_t$.
	Since the cooperators won't die in the \textbf{Ghost system} we have $\forall \sigma \in \mathcal{C} $ $\forall i \in N_{\mathfrak{C}}(\sigma)$ $\overline{\tau}_1^i < +\infty$ $\G_\sigma$ a.s. also we have the following upper bounding lemma. The proof is the same as Lemma \ref{Lemme Majoration uniforme Extinction}, the only modification is to change $T_n = \lbrace \tau_1 = n \rbrace$ by $T_n = \left\lbrace \max\limits_{i \in N_\mathfrak{C}(\sigma^0)} \tau_1^i = n \right\rbrace$
	
	\begin{lemme}\label{Lemme Maj unif Coexistence}
			There is $\overline{\mathfrak{m}} \in \N$ and $\overline{\varepsilon} >0$ such that for each configuration $\sigma \in E$ with at least two cooperators and all $s\in \R$ we have: 
				\[
				\G_\sigma\left(\max\limits_{i \in N_\mathfrak{C}(\sigma)}{\tau}_1^i > s\right) \leq \left(1 - \overline{\varepsilon} \right)^{\lfloor s/\bar{\mathfrak{m}}\rfloor }
				\]
	\end{lemme}
	We denote $G$ the expectation using the probability measure $\G$. We look at the wealth of one individual then the birth it gives, make its wealth decrease, then the upper bounding of the variation of wealth has to take it into account. Using the same argument as those in the previous theorem, we prove that: $\exists \nu >0$ such that: $\forall n \in \N$ $\forall \sigma \in E$ with at least two cooperators 
		\[
		G_\sigma\left({\mathfrak{C}}^i_{{\tau}^i_{n+1}} - {\mathfrak{C}}^i_{{\tau}^i_n}\right) \geq \mathbf{R} - \nu (\mathbf{S} + w_0):= \alpha >0
		\]
		
		Also for all $n \in \N$ we have: (because $\mathbf{R} < \mathbf{S + w_0}$)
		\[
		|{\mathfrak{C}}^i_{{\tau}^i_{n+1}} - {\mathfrak{C}}^i_{{\tau}^i_n}| \leq (\mathbf{S} + w_0) \tau^i
		\]
		Then using Lemma \ref{Lemme Maj unif Coexistence} we can apply Corollary \ref{Corol Divergence a.s.} to get: for a fixed $0 < \delta < \alpha$ $\forall i \in N_\mathfrak{C}(\sigma)$ $\G_\sigma$ almost surely:
		\[
		{\mathfrak{C}}^{i}_{{\tau}^i_{n}} > \delta \tau^i_n \longrightarrow +\infty
		\]
		
		Then intersecting those events we have for a fixed $0 < \delta < \alpha$,  $\G_\sigma$ almost surely: 
				 \begin{eqnarray}\label{Divergence vers infty}
				\forall i\in N_\mathfrak{C}(\sigma),\qquad {\mathfrak{C}}^{i}_{{\tau}^i_{n}} > \delta \tau^i_n  \longrightarrow +\infty.
				 \end{eqnarray}
	
		\vspace{15 pt}
		Now let us prove that $\min\limits_{i \in N_\mathfrak{C}(\sigma)}{\mathfrak{C}}^{i}_{n}$ cannot go under 0 infinitely often. Because we have $\forall n\in \N$ $\forall k \in \lbrace {\tau}^i_n ,\dots,{\tau}^i_{n+1}\rbrace$   ${\mathfrak{C}}^{i}_{k} > {\mathfrak{C}}^{i}_{{\tau}^i_n} - (\mathbf{S}+w_0)({\tau}^i_{n+1} - {\tau}_n)$ and because ${\mathfrak{C}}^{i}_{{\tau}^i_n}>\delta {\tau}^i_n \geq n$ we have for all $n\in \N$:
		
\begin{align*}
		\lbrace \exists k\in \lbrace \tau^i_n, \dots,\tau^i_{n+1}\rbrace , \mathfrak{C}^i_k \leq 0 \rbrace &\subset \left\lbrace \delta \tau_i^n \leq (\mathbf{S} + w_0) (\tau_{n+1}^i - \tau_n^i) \right\rbrace \\& \subset \left\lbrace \min_{j \in N_\mathfrak{C}(\sigma)} \delta \tau^j_n \leq (\mathbf{S} + w_0) \max_{j \in N_\mathfrak{C}(\sigma)} (\tau^j_{n+1} - \tau_n^j)\right\rbrace \\& \subset \left\lbrace  \delta n \leq (\mathbf{S} + w_0) \max_{j \in N_\mathfrak{C}(\sigma)} (\tau^j_{n+1} - \tau_n^j)\right\rbrace
\end{align*}
		Hence we get: for all $n\in \N$
		
		\begin{align*}
		\G_\sigma\left(\min_{i \in N_\mathfrak{C}(\sigma)} \min_{{\tau}^i_n \leq k \leq {\tau}^i_{n+1}} {\mathfrak{C}}^{i}_{k} \leq 0\right) & \leq K \G_\sigma\left(\max_{i \in N_\mathfrak{C}(\sigma)}({\tau}^i_{n+1} - {\tau}^i_n)(\mathbf{S} + w_0) > \delta n\right) \\
		& \leq K (\mathbf{S} + w_0)^2 \frac{1}{\delta^2 n^2} G_\sigma\left(\max_{i \in N_\mathfrak{C}(\sigma)}({\tau}^i_{n+1} - {\tau}^i_n)^2\right)
		\end{align*}
		which is summable using Lemma \ref{Lemme Maj unif Coexistence}.
		Then we have with (\ref{Divergence vers infty}) \begin{eqnarray}\label{Passage sous 0 infiniment souvent}
		\G_\sigma\left( \min_{i \in N_\mathfrak{C}(\sigma)}{\mathfrak{C}}^{i}_{n} \leq 0 \ i.o\right) = 0.
		\end{eqnarray}
		We do the same reasoning with the defectors. We have $\exists \nu' >0$ and for all $i\in N_\mathfrak{D}(\sigma)$ there exists $(\tilde{\tau}_n^i)$ such that there is $\tilde{\mathfrak{m}} \in \N$ and $\tilde{\varepsilon}>0$ such that for each configuration $\sigma \in E$ with at least one cooperator and one defector and all $s\in \R$ we have:
		\[
		\mathbb{G}_\sigma\left(\max_{i\in N_\mathfrak{D}(\sigma)} \tilde{\tau}_1^i > s \right) \leq (1 - \tilde{\varepsilon})^{\lfloor s/ \tilde{\mathfrak{m}} \rfloor}
		\]
		
		Then we get for a fixed $0< \delta'<\mathbf{T} - \nu' (2\mathbf{P}+w_0)$, $\mathbb{G}_\sigma$ almost surely:
	
		\begin{eqnarray}\label{Divergence vers infty pour defector}
		\forall i \in N_\mathfrak{D}(\sigma), \qquad \mathfrak{D}^i_{\tilde{\tau}^i_n} > \delta' \tilde{\tau}^i_n \longrightarrow +\infty 
		\end{eqnarray}
	
		and also we have:
		
	\begin{eqnarray}\label{Passage sous 0 infiniment souvent pour defector}
		\mathbb{G}_\sigma\left(\min_{i \in N_\mathfrak{D}(\sigma)} \mathfrak{D}^i_n \leq 0 \ i.o. \right) = 0
	\end{eqnarray}

		Let $D = \min \left(\min\limits_{i \in N_\mathfrak{C}(\sigma)}\min\limits_{n \in \N} {\mathfrak{C}}^{i}_n, \min\limits_{i \in N_\mathfrak{D}(\sigma)}\min\limits_{n \in \N} {\mathfrak{D}}^{i}_n\right)$, using (\ref{Divergence vers infty}), (\ref{Passage sous 0 infiniment souvent}), (\ref{Divergence vers infty pour defector}) and (\ref{Passage sous 0 infiniment souvent pour defector}) we have for all $\sigma \in \mathcal{C}$ $\G_\sigma(D = -\infty) = 0$. A a consequence for $0<p\leq 1$ fixed, $\exists$ $L \in \R$ such that:
		\[
		\G_\sigma(D \leq L) \leq p
		\] 
		Let us denote $\sigma + L$ the set of configuration such that for all $\bar{\sigma} \in \sigma + L$ $\mathfrak{p}(\sigma) = \mathfrak{p}(\bar{\sigma})$ and the wealth of every individual with positive wealth in ${\sigma}$ is increased by at least $|L|$ in $\bar{\sigma}$. Then if in $\sigma$ an individual has a wealth of $k > 0$, in configurations of $\sigma + L$ he has a wealth of $k + |L|$. As a consequence we have $N_\mathfrak{C}(\bar{\sigma}) = N_\mathfrak{C}(\sigma)$ and $N_\mathfrak{D}(\bar{\sigma}) = N_\mathfrak{D}(\sigma)$.\\
		Since in the \textbf{Ghost system}, the evolution of wealth only depends on the photographs we have: for all $\bar{\sigma} \in \sigma + L$
		\[
		\G_{{\sigma}} (\forall n\in \N \  \forall i\in N_\mathfrak{C}(\sigma) \cup N_\mathfrak{D}(\sigma),\  Y^i_n > 0 ) \geq 1 - p
		\]
		Let $n_0 \in \N$ such that:
		\[
		\G_\sigma(\forall n \in \lbrace 0,\dots,n_0\rbrace,\forall i \in N_\mathfrak{C}(\sigma) \cup N_\mathfrak{D}(\sigma),\  Y_n^i >0 \text{ and } {\sigma}_{n_0}\in \sigma + L  ) > 0
		\]
		We denote $\lbrace \forall n \in \lbrace 0,\dots,n_0\rbrace,\forall i \in N_\mathfrak{C}(\sigma)\cup N_\mathfrak{D}(\sigma),\  Y_n^i >0 \text{ and } {\sigma}_{n_0}\in \sigma + L  \rbrace :=  E_{n_0}(\sigma + L)$. We have using Markov property: $\G_\sigma(\forall n \in \N, \forall i \in N_\mathfrak{C}(\sigma)\cup N_\mathfrak{D}(\sigma),\  Y_n^i >0)$ is equal to:
	 \[
	 \G_\sigma(E_{n_0}(\sigma +L)) \G_{\sigma +L}(\forall n \in \N,\forall i \in N_\mathfrak{C}(\sigma + L)\cup N_\mathfrak{D}(\sigma),\  Y_n^i >0) >0
	 \]

		Since in $\lbrace \forall n \in \N, \forall i \in N_\mathfrak{C}(\sigma)\cup N_\mathfrak{D}(\sigma),\  Y_n^i >0 \rbrace$ cooperators always have positive wealth, on this event the \textbf{Ghost system} is just the \textbf{True system} and we get:
		\[
					\PP_\sigma (\forall n \in \N, \forall i \in N_\mathfrak{C}(\sigma)\cup N_\mathfrak{D}(\sigma),\  Y_n^i >0) > 0
		\]
\end{proof}

\section{Proof on the Mean Field system}
\subsection{Proof of Theorem \ref{thm Modele simplifié Garanti}}
In this part we will prove Theorem \ref{thm Modele simplifié Garanti}

	Let us begin the \textit{Proof}.\vspace{5pt}\\ 
	 Using the Law of Large Numbers we have for all $q_0 >0$,  $\PP_{q_0}$ a.s. for all $t>0$ $\exists N \in \N$ such that $t_N \leq t < t_{N+1}$
	  	  \[
	  	  \frac{1}{N} \overline{\mathfrak{C}}(t) = \frac{q_0}{N} + \frac{1}{N} \sum_{n=1}^{N} \overline{U}_\mathfrak{C}(t_n^\mathfrak{C}) \underset{N \to +\infty}{\longrightarrow} \beta^0  \mathbf{R} - \rho^0 \mathbf{S}
	  	  \]
	  	  Then if $\beta^0 \mathbf{R} - \rho^0 \mathbf{S} >0$ we have $\PP_{q_0}$ a.s. $\overline{\mathfrak{C}}(t) \underset{t \to \infty}{\longrightarrow} +\infty$ and since the increasing are bounded, as a consequence for all $\eta >0$ there exists $\overline{M} \in \mathbf{R}\Z + \mathbf{S}\Z$ such that:
	  	  \[
	  	  \PP_{q_0}\left( \min_{t\in \R_+} \overline{\mathfrak{C}}(t) > \overline{M}\right) \geq 1 - \eta
	  	  \]
	  	  Since $\overline{C}$ is a Levy process \emph{i.e.} since the increments of $\overline{\mathfrak{C}}$ don't depend on where $\overline{\mathfrak{C}}$ is, we have:
	  	  \[
	  	  \PP_{q_0 + |\overline{M}|} \left(\min_{t \in \R_+} \overline{\mathfrak{C}}(t)> \underbrace{\overline{M} + |\overline{M}|}_{>0}\right) \geq 1 - \eta
	  	  \]
	  	  Let $T>0$ be such that \[
	  	  \PP_{q_0}\left(\min_{t \leq T} \overline{\mathfrak{C}}(t) >0, \overline{\mathfrak{C}}(T) = q_0 + |\overline{M}|\right) >0
	  	  \]
	  	  Hence we get using Markov property:
	  	  
	\[  	 
	 \PP_{q_0}\left(\min_{t\in \R_+} \overline{\mathfrak{C}}(t)> 0\right) \geq \underbrace{\PP_{q_0} \left(\min_{t\leq T} \overline{\mathfrak{C}}(t)> 0, \overline{\mathfrak{C}}(T) = q_0 + |\overline{M}|\right)}_{>0} \underbrace{\PP_{q_0 + |\overline{M}|}\left(\min_{t\in \R_+} \overline{\mathfrak{C}}(t)> 0\right)}_{\geq 1 - \eta} >0
	\]
	 	 \qed

\subsection{Proof of Proposition \ref{Prop Concentration Wealth MF DPD}}

	  	  Firstly let us notice that $(\overline{\mathfrak{C}}(t))_t$ only take values in $q_0 + \mathbf{R}\Z + \mathbf{S}\Z$.
	  	  Then we get:
	  	  
	  	  \begin{align*}
	  	  \dro \E(\overline{\mathfrak{C}}(t)) &= \sum_{k,\ell \in \Z} (q_0 + k \mathbf{R} - \ell \mathbf{S}) \dro \PP(\overline{\mathfrak{C}}(t) = q_0 + k \mathbf{R} - \ell \mathbf{S}) \\
	  	  &=   {v} (\beta^0\mathbf{R} - \rho^0 \mathbf{S})
	  	  \end{align*}
	  	  
	  	  Let us compute for all $t>0$ the variance of $\overline{\mathfrak{C}}(t)$.
	  	  \begin{align*}
	  	  \dro \E(\overline{\mathfrak{C}}(t)^2) &= \sum_{k,\ell \in \Z } (q_0 +k\mathbf{R} -\ell \mathbf{S})^2 \dro \PP(\overline{\mathfrak{C}}(t) = q_0 + k \mathbf{R} - \ell \mathbf{S}) \\
	  	   & = \sum_{k,\ell \in \Z } (q_0 + (k-1)\mathbf{R} - \ell \mathbf{S} + \mathbf{R})^2 \beta^0 \PP(\overline{\mathfrak{C}}(t) =q_0 + (k-1)\mathbf{R} - \ell \mathbf{S}) \\
	  	    & \hspace{10pt} + \sum_{k,\ell \in \Z} (q_0 + k \mathbf{R} -(\ell -1)\mathbf{S} - \mathbf{S})^2 \rho^0 \PP(\overline{\mathfrak{C}}(t) = q_0 +k\mathbf{R} - (\ell - 1)\mathbf{S}) \\
	  	    &\hspace{10pt} - \E(\overline{\mathfrak{C}}(t)^2) \\
	  	    &= \beta^0 \mathbf{R}^2 + \rho^0 \mathbf{S}^2 + 2 \E(\overline{\mathfrak{C}}(t)) \dro \E(\overline{\mathfrak{C}}(t))\\
	  	    &= \beta^0 \mathbf{R}^2 + \rho^0 \mathbf{S}^2 + \dro \E(\overline{\mathfrak{C}}(t))^2
	  	  \end{align*}
	  	  
	  	  Then we have: $\forall t>0$
	  	  \begin{eqnarray}
	  	  \V(\overline{\mathfrak{C}}(t)) = \left(\beta^0 \mathbf{R}^2 + \rho^0 \mathbf{S}^2\right)t
	  	  \end{eqnarray}
	  	  Then using the Chebyshev inequality we have the following concentration inequality: for all $\varepsilon_t >0$
	  	  
	  	  \begin{eqnarray}\label{Inegalité Controle Fantome Tchebychev}
	  	  \PP(|\overline{\mathfrak{C}}(t) - \E(\overline{\mathfrak{C}}(t))| > \varepsilon_t) \leq {v} \frac{t}{\varepsilon_t^2} \left(\beta^0\mathbf{R}^2 + \rho^0\mathbf{S}^2\right).
	  	  \end{eqnarray}
	  	  
	  	  We now take for all $\eta >0$, $\varepsilon^\eta_t = \eta \sqrt{\frac{v}{2}t\left(\beta^0\mathbf{R}^2 + \rho^0\mathbf{S}^2\right)}$ and obtain (\ref{eqn tchebychev MF DPD}).\\
	  	  
	  	  Moreover we notice that 
	  	   \[\E(\overline{\mathfrak{C}^\eta}(t)) - \varepsilon_t^\eta = q_0+t(\beta^0(1 - {\eta^{-2}})\mathbf{R}-\rho^0 \mathbf{S})-\eta \sqrt{{v}(\beta^0\mathbf{R}^2+\rho^0\mathbf{S}^2)}\sqrt{t}.\] 
	  	  	 The minimum of $t \mapsto q_0 + t(\beta^0\mathbf{R} - \rho^0 \mathbf{S}) - \eta\sqrt{{v}(\beta^0\mathbf{R}^2+\rho^0\mathbf{S}^2)} \sqrt{t}$ is reached in \[t = \frac{v\eta^2 (\beta^0\mathbf{R}^2+\rho^0\mathbf{S}^2)}{4(\beta^0 \mathbf{R} - \rho^0 \mathbf{S})^2}.\]  This minimum is $q_0 - \frac{v \eta^2 ((\beta^0\mathbf{R}^2+\rho^0\mathbf{S}^2))}{4(\beta^0\mathbf{R} - \rho^0 \mathbf{S})}.$ Hence for all $\tau$ such that:
	  	  		  \[ \tau +  q_0 - \frac{\eta^2 v (\beta^0\mathbf{R}^2+\rho^0\mathbf{S}^2)}{4(\beta^0 \mathbf{R} - \rho^0 \mathbf{S})} > 0 \]
	  	  	We have: 
	  	  	\[
	  	  	\PP_{q_0} \left(\overline{\mathfrak{C}}(t) \geq \tau \right) \geq 1 - \eta^{-2}.
	  	  	\]
	  	  	  
\section*{Acknowledgment}
I would like to thank my advisors for the review Laurent Miclo and J\'er\^ome Renault. Thanks to Xavier Bressaud for initiating me into the world of Research and to giving me this topic.

\bibliographystyle{apalike}
\bibliography{main}

\end{document}